\documentclass{aptpub}

\newtheorem{de}{Definition}[section]
\usepackage[scr]{rsfso}
\usepackage{booktabs}
\usepackage{subfig}
\usepackage{placeins}
\usepackage{color}
\usepackage{bm}
\usepackage{natbib}
\usepackage{graphicx}

\def\d{\delta}
\def\c{\circ}

\def\E{\mathbb{E}}

\def\C{\mathbb{C}ov}
\def\P{\mathbb{P}}
\def\R{\mathbb{R}}
\def\N{\mathbb{N}}

\def\D{\mathrm{d}}
\def\Var{\mathbb{V}ar}
\def\C{\mathbb{C}ov}
\def\Cov{\mathbb{C}ov}

\DeclareMathOperator*{\argmax}{arg\,max}
\DeclareMathOperator*{\argmin}{arg\,min}
\DeclareMathOperator*{\sign}{sign}

 \DeclareMathOperator*{\support}{supp}

\DeclareMathOperator*{\for}{\quad\text{for}\quad}

\newcommand{\mmod}{\ \mathrm{mod}\ }

\newcommand\ind[1]{\mathbb{I}_{#1}}

\usepackage[colorlinks = true,linkcolor = blue,urlcolor  = blue,citecolor = blue,anchorcolor = blue]{hyperref}
\authornames{C.~BASTIAN {\it et al.}} % insert the authors here for use in running head. If three or more authors please use (for example) M.~YARROW {\it et al.} Author names should follow the same M.~YARROW format and if two authors, separate by 'AND'.
\shorttitle{Orthogonal Dice} % insert short title here for use in running head

\begin{document}%\recd{}{}%Do not alter this line.

\title{Orthogonal Dice} % insert title

\authorone[Princeton University]{Caleb Deen Bastian} 
%Affiliation is just the name of your university or institution, for example 'University of Sheffield'. Author names should be of the form 'Mark Yarrow'. 
%Authors should be ordered alphabetically subject to the convention in that particular authors country. For example 'Remco van der Hofstad' would be listed under 'H' as is standard in the Netherlands. 

%Please use the following format for addresses and emails. The APT office will sort this out after you submit your files.
\addressone{Program in Applied and Computational Mathematics, Fine Hall, Washington Road, Princeton, NJ. 08544. USA} % Your postal address goes here.
\emailone{cbastian@princeton.edu} %Authors email goes here.
\authortwo[Princeton University]{Herschel Rabitz}
\addresstwo{Department of Chemistry, Frick Laboratory, Princeton, NJ. 08544. USA}
\authorthree[The Ohio State University]{Grzegorz Rempala}
\addressthree{Department of Mathematics, 231 W 18th Ave, Columbus, OH. 43210. USA}

\begin{abstract}
In this paper, we introduce a family of discrete rectangular uniform distributions on the natural numbers—referred to as {\em orthogonal dice}—characterized by the property that their means equal their variances. These distributions arise naturally in statistics and applied mathematics. We show that the orthogonal dice correspond to solutions of a quadratic Diophantine equation on the naturals, exhibiting divisibility properties tied to their dimensions, generating coprime arithmetic progressions, yielding disjoint partitions of the naturals, and displaying self-similarity. Their associated random counting measures (mixed binomial processes) exhibit interesting structural properties, including orthogonal splitting and convergence to Poisson limits. As a result, the orthogonal dice define canonical stochastic processes that that may be used to construct Brownian and geometric Brownian motions.  More broadly, they serve as Poisson-like building blocks—natural substrates for modeling systems with bounded counts. Furthermore, they induce a trichotomy within the broader class of such distributions, partitioning them into three infinite subfamilies—negative, orthogonal, and positive—according to their mean-variance relationships.
\end{abstract}

%insert keywords separated by a semicolon. You should avoid including keywords which appear in the title.

    % insert the primary 2020 Maths Subject Classification number in the first bracket e.g. \ams{60E20}{49G03; 49F10}
                 % and the secondary ams number(s) in the second bracket
                 %Maximum of three in each, ideally one or two in each primary and secondary.
                 %codes found here ``https://mathscinet.ams.org/msnhtml/msc2020.pdf''

\section{Introduction} 

% [insert literature on discrete uniform distributions; describe relation to other families of distributions; describe mean-variance equality literature and ideas]

The variance-to-mean ratio (VMR)---also known as the Fano factor \citep{fano}, index of dispersion \citep{Barlow}, dispersion index, coefficient of dispersion, relative variance, etc.---is a fundamental quantity in statistics, e.g., count data. The Poisson-type family of distributions---binomial, Poisson, and negative binomial---are well-known probability counting measures that yield respectively  under-dispersed (VMR$<$1), equi-dispered (VMR=1), over-dispersed (VMR$>$1) random variables. Two of these distributions, Poisson and negative binomial, are supported on the naturals $\N$ whereas binomial is supported on the interval of consecutive naturals $\{0,1,\dotsb,n\}$. 

Consider a probability counting measure $\kappa$. Let $K\sim\kappa$ be a count random variable. Most commonly encountered counting laws, e.g., Poisson, geometric, zeta, and so on, are defined on infinite intervals, i.e., $\N_{\ge0}$ or $\N_{\ge1}$. These laws are fundamental and often represent foci of convergence for random sums or other operations, e.g., see Billingsley's book \citep{billingsley}. However, in practical settings, the Poisson assumption can be questioned, for instance for traffic flows, on the principle that $P(K>c)>0$ for any positive constant $c$, such as greater than the number of cars on Earth. Therefore, it is incumbent to think of distributions that natively model bounded counts, where $m\le K\le n$ with $K\in\{m,\dotsb,n\}$. The simplest distribution on $\{m,\dotsb,n\}$ is the maximum entropy discrete rectangular uniform $\kappa_{m,n}$ defined by $\P(K=k)=1/(n-m+1)$ for $k\in\{m,\dotsb,n\}$. We shall be concerned with characterization of the space of such distributions relative to the VMR, rearranged into the difference between the variance and the mean (`variance-to-mean difference'). These distributions are primitives to many other families of counting laws.

Modeling bounded count data, the discrete rectangular uniform distribution on the naturals is an essential law in statistics \citep{discrete}. By construction, these distributions, termed `dice,' yield equally likely outcomes on their support of an interval of consecutive naturals. While dice are mainstays in statistics, it appears the notion of their classification has received less attention. In this article, we characterize the space of dice relative to the notion of mean-variance equality, prominently seen with the Poisson distribution, where dice possessing mean-variance equality are referred to as `orthogonal,' under-dispersed dice are referred to as `negative', and over-dispersed dice as `positive.'  Each class contains infinite members. While simple, the classification scheme based on mean-variance equality is inspired and based on the mean-variance structure of random counting measures, specifically mixed binomial processes, formed by such dice. There, random counting measures whose probability counting measures exhibit mean-variance equality are orthogonal: counts are de-correlated random variables in disjoint subsets. Thus, the `orthogonal dice' immediately possess orthogonal splitting, similar to the Poisson random measure \citep{kingman-poisson-processes,cinlar}. It turns out, the `orthogonal dice' also possess a Poisson limit theorem, where ranging over the dice and thinning yields a Poisson random measure. Dice mean-variance equality yields a quadratic Diophantine equation \citep{diophantine,HW79}, whose solutions are precisely the support parameters of the orthogonal dice. Therefore, the `orthogonal dice' also inherit number theoretic properties with relevance to applied mathematics, including divisibility properties of the sizes of the orthogonal dice, coprime arithmetic progressions, disjoint partitions of the naturals, and self-similarity. This article focuses on the essentials of the orthogonal dice, such as existence, uniqueness, representations, etc. It also gives some illustrative applications to statistical problems of modeling bounded count data, construction of stochastic processes, etc. The objective is to introduce and examine mean-variance equality-based classification of the family of dice through the foundational notion of orthogonal dice and to highlight how this obtained elementary classification of dice based on VMR is non-trivial and useful.

The paper is organized as follows. In Section~\ref{sec:back}, we provide a brief review of the necessary stochastic background, highlighting the convergence of discrete uniform distributions to the Poisson law. Additionally, we delve into the mathematical foundations of this convergence. Next, in Section~\ref{sec:construct}, we establish the existence and uniqueness of an infinite family of discrete uniform random measures. These measures exhibit Poisson convergence with vanishing covariance in disjoint subspaces, which we refer to as `orthogonal dice.' We thoroughly investigate the properties of these orthogonal dice. In Section~\ref{sec:applications} we give some applications of the orthogonal dice, both theoretical and applied. There we study application of `dice' to modeling systems having bounded count data, conduct inference of bounded counts, construct the Wiener and geometric Brownian motion processes from the orthogonal dice, model traffic flows, and construct sparse lattices of naturals from the self-similarity of the orthogonal dice. Lastly in Section~\ref{sec:dis}, we conclude with a concise discussion and present our key findings and implications.

\section{Background}\label{sec:back} We give the necessary background by describing random counting measures, in particular the mixed binomial process, and convergence. %All processes considered in this article are mixed binomial processes. 

\subsection{Random counting measures}
Let $(E,\mathscr{E})$ be a measurable space and let $\nu$ be a probability measure on it. Let  $\mathbf{X}=\{X_i\}$ be an independency (collection) of (iid) $E$ valued random variables with law $\nu$. Let $K\sim\kappa$ be a $\N_{\ge 0}$-valued random variable, distributed according to probability counting measure $\kappa$, independent of $\mathbf{X}$, with mean $c>0$ and variance $\delta^2\ge0$. The \emph{mixed binomial process} is identified to the pair of deterministic probability measures $N=(\kappa,\nu)$ on $(E,\mathscr{E})$ through \emph{stone throwing construction} \citep{cinlar,Bastian:2020tb,rm} (STC) as \begin{equation}\label{eq:stc}N(A)=N\ind{A}=\int_{E}N(\D x)\ind{A}(x)\equiv\sum_{i}^{K}\ind{A}(X_i)\quad\text{for}\quad A\in\mathscr{E}\end{equation} where $\ind{A}(x)$ is an indicator function taking the value one if $x\in A$ and zero otherwise. The mixed binomial process is also sometimes referred to as a \emph{proper point process}. Because $\ind{A}(x)=\delta_x(A)$ where $\delta_x$ is the Dirac measure sitting at $x\in E$, STC of the mixed binomial process may be concisely written as \[N=\sum_{i}^K\delta_{X_i}\] with independency $\{K,X_1,X_2,\dotsb\}$. We denote $\mathscr{E}_{\ge0}$ the set of non-negative $\mathscr{E}$-measurable functions. Recall that the law of $N$ is uniquely determined by the \emph{Laplace functional} $L$ \begin{equation}\label{eq:laplace} L(f)=\E e^{-Nf} = \psi(\nu e^{-f})\quad\text{for}\quad f\in\mathscr{E}_{\ge0}\end{equation} where $\psi$ is the \emph{probability generating function} (pgf) of $K$. %The almost sure finiteness of $Nf$ is determined by the Laplace functional \[\P(Nf<+\infty)=\lim_{q\downarrow0}L(qf)\]

For $f\in\mathscr{E}_{\ge0}$, the \emph{mean} and \emph{variance} of the random variable $Nf$ (if they exist) are \begin{align}\E Nf &= c\nu f\label{eq:mu0}\\\Var Nf &= c\nu f^2 + (\delta^2-c)(\nu f)^2.\label{eq:var0}\end{align} The mean measure $\lambda=c\nu$ is also known as the \emph{intensity measure} of $N$. The variance \eqref{eq:var0} contains two terms depending on the mean and variance of $\kappa$ and the moments of $f$. The first term involves the mean (first moment) of $\kappa$ and the second moment of $f$, whereas the second term depends on the difference between the variance and mean of $\kappa$ and the square of the first moment of $f$. For example, for Poisson, we have $c=\delta^2$ such that $\Var Nf = c\nu f^2$, whereas for Dirac we have $\delta^2=0$ such that $\Var Nf = c(\nu f^2 - (\nu f)^2) = c\Var f$. %Please see \cite{Bastian:2020tb} for more information on these relations relative to the Poisson-type family of random counting measures. 

For arbitrary $f,g\in\mathscr{E}_{\ge0}$, we have \emph{covariance} \begin{equation}\label{eq:cov0}\Cov(Nf,Ng)=c\nu(fg)+(\delta^2-c)\nu f\nu g.\end{equation} The mean-covariance structure of $N$ is thus determined by the mean and variance of $\kappa$. The \emph{moments} of $Nf$ (if they exist) can be obtained from the Laplace functional \[ \E(Nf)^n = (-1)^n\lim_{q\downarrow0}\frac{\partial^n}{\partial q^n}L(qf)\quad\text{for}\quad n\in\N_{\ge1}.\] 

Consider measurable subset (`subspace') $A\subseteq E$ with $\nu(A)=a>0$ and the trace random measure defined as $N_A(B)=N(A\cap B)=(N\ind{A})(B)$ on $(E,\mathscr{E})$. Thus $N_A\equiv N\ind{A}$.  We also identify the random measure $N_A$ as a pair $N_A=(N\ind{A},\nu_A)$ where $\nu_A(B)=\nu(A\cap B)/\nu(A)$, according to STC notation $N=(\kappa,\nu)$. The law of the random measure $N_A$ for $A\subseteq E$ with $\nu(A)=a>0$ is given by its Laplace functional \begin{equation}\label{eq:la1} L_A(f) =\psi_A(\nu_A e^{-f})= \psi(a\nu_A e^{-f}+1-a)\quad\text{for}\quad f\in\mathscr{E}_{\ge0}\end{equation} where $\psi_A$ is pgf of $N\ind{A}$ and $\psi$ is pgf of $K=N\ind{E}$. The mass function of the random variable $N(A)=N\ind{A}$ is given by \begin{equation}\label{eq:pdf}\P(N(A)=k)=\psi_A^{(k)}(0) / k! = [t^k]\psi_A(t)=\frac{1}{2\pi i}\oint_{C}\frac{\psi_A(t)}{t^{k+1}}\D t\for k\in\N_{\ge0}\end{equation} where $\psi_A$ \eqref{eq:la1} and the contour integral is over the unit circle $C$. This is stably obtained by computing the residue of $\psi_A(t)/t^{k+1}$ at the point $0$.

\subsection{Coprimality}

Recall that any two positive naturals $a$ and $b$ are \emph{coprime} if their greatest common divisor is 1, indicated by $(a,b)\equiv\gcd(a,b)=1$. Said another way, they share no common \emph{prime} factors. 

\subsection{Review of Convergence}\label{sec:converge} We give a brief review of weak convergence results for the discrete uniform distribution.  The first basic convergence result is that the probability generating function of the discrete uniform distribution uniformly converges to Poisson while ranging over degenerating limiting support and thinning, the Poisson limit theorem (PLT). This is a special case of classical results on independent thinnings by J. Mecke. For convenience we give the proof.

\begin{samepage}
\begin{thm}[PLT]\label{thm:main} Let $\psi_{m,n}(t)$ be the probability generating function for a discrete uniform distribution supported on the set of consecutive integers $m,\dotsc,n$. Also, for any $a\in[0,1]$ let \[\psi_{m,n}^a(t) = \psi_{m,n}(at+1-a).\] If $m/n\rightarrow1$ and $na\rightarrow b>0$ as $m,n\rightarrow\infty$, $a\rightarrow0$, then \[\sup_{t\in[0,1]}|\psi_{m,n}^a(t)-z_b(t)|\rightarrow0\] where $z_b$ is the probability generating function of a random variable $Poisson(b)$.
\end{thm}\end{samepage}
\begin{proof} Since the Poisson variable is uniquely defined by its moments, it suffices to show convergence of all the moments, in order to argue weak convergence from which the above convergence of pgfs will follow (see e.g. Billingsley’s book \citep{billingsley}). Further it suffices to show the convergence of all factorial moments. The factorial moments of the thinned uniform variable $X$ (i.e., the one with pgf $\psi_{m,n}^a(t)$) are given by \[\E[X(X-1)\dotsb(X-k+1)] = \mu_k(m,n,a)=a^k D_k[\psi_{m,n}^a(t)]_{|t=1}\] where $D_k[\cdot]$ denotes the $k$-th derivative. Note that \[\psi_{m,n}(t) = (n-m+1)^{-1}\sum_{i=m}^n t^i\] and \[D_k[\psi_{m,n}^a(t)]_{|t=1} = a^k\frac{n(n-1)\dotsb(n-k+1)\dotsb m(m-1)\dotsb(m-k+1)}{n-m+1}\] that satisfies \[a^km^k\le D_k[\psi_{m,n}^a(t)]_{|t=1}\le a^k n^k\] which implies in view of the assumptions that \[D_k[\psi_{m,n}^a(t)]_{|t=1}\rightarrow b^k.\] But we note that $\{b^k\}_{k=1}^\infty$ is a sequence of factorial moments of a Poisson random variable with mean $b$.
\end{proof}

Let $N_{m,n}=(\kappa_{m,n},\nu)$ be a discrete uniform random measure and recall $\psi_{m,n}$ is the pgf of $\kappa_{m,n}$. Let $N_{m,n}^a=(\kappa_{m,n}^a,\nu_A)$ be the restriction of the die random measure $N_{m,n}$ to subset $A\subseteq E$ with mass $\nu(A)=a>0$ and law $\nu_A(B)=\nu(A\cap B)/\nu(A)$. $\kappa_{m,n}^a$ has pgf $\psi_{m,n}^a(t)=\psi_{m,n}(1-a+at)$ following from binomial thinning of the probability counting measure, i.e., $\psi_{m,n}^a(t) = \psi_{m,n}\circ\psi_a(t)$ where $\psi_p(t)=1-p+pt$ is the pgf of a $\text{Bernoulli}(p)$ random variable. In this sense $N_{m,n}^a$ is the $a$-thinned discrete uniform random measure $N_{m,n}$. The second convergence result which immediately follows is weak convergence of the thinned random counting measures.

\begin{thm}[Random measure PLT]\label{thm:main2}Assume the hypotheses of Theorem~\ref{thm:main}. Then the sequence of thinned discrete uniform random measures $(N_{m,n}^a)$ converges in distribution (converges weakly) to the Poisson random measure $N$, that is, \[\lim_{\substack{m,n\rightarrow\infty\\a\rightarrow0}}\E e^{-N_{m,n}^af}=\E e^{-Nf}\quad\text{for}\quad f\in\mathscr{E}_{\ge0}.\]
\end{thm}
\begin{proof} The result follows from the unique determination of law by the Laplace functional \eqref{eq:laplace} and Theorem~\ref{thm:main}.
\end{proof}

\section{The Orthogonal Dice}\label{sec:construct} In this section we construct a subsequence of dice whose means equal their variances (`orthogonality'). This simple property turns out to confer a deep structure to the orthogonal dice. We briefly discuss orthogonality and move to construction of the orthogonal dice. 
\subsection{Orthogonality}\label{sec:ortho}

The fundmental notion of de-correlation, or orthogonality, of random measures is defined below. 

\begin{de}[Orthogonality] Let $N$ be a random measure on $(E,\mathscr{E})$. It is said to be \emph{orthogonal} if $N(A),\dotsb,N(B)$ are decorrelated for all choices of finitely many disjoint sets $A,\dotsb,B$ in $\mathscr{E}$.  
\end{de}

Orthogonality of the mixed binomial process is encoded by the probability counting measure $\kappa$, where the mean equals the variance.

\begin{prop}[Orthogonality]\label{de:ortho} The mixed binomial process $N=(\kappa,\nu)$ is orthogonal iff $c=\delta^2$, which implies for orthogonal $N$ and arbitrary $f,g\in\mathscr{E}_{\ge0}$ that $\Cov(Nf,Ng)=c\nu(fg)$.\end{prop} 

 A canonical orthogonal random measure is Poisson. The structure of orthogonality yields a splitting property. 

\begin{rem}[Splitting] The proposition conveys an orthogonal splitting property such that the variance decomposes as \begin{align*}\text{for disjoint }A,\dotsb,B \text{ in }\mathscr{E}:&\\\Var (N(A)+\dotsb+ N(B)) &= \Var N(A) +\dotsb + \Var N(B).\end{align*}
\end{rem}

\subsection{Discrete uniform}\label{sec:disc}

The discrete uniform distribution is the object of focus of this article. We briefly give some basic properties of such distributions.

The set of permissible supports for the discrete uniform distribution is defined as \[B\equiv\{(m,n): \text{integers }0\le m\le n\text{ except }m=n=0\}.\] Consider the collection of discrete (rectangular) uniform family of distributions $\mathcal{K}$ \begin{equation}\label{eq:U}\mathcal{K}=\{\kappa_{m,n}=\text{Uniform}\{m,m+1,\dotsb,n-1,n\}: (m,n)\in B\}\end{equation} where $\kappa_{m,n}$ has mean $c=(m+n)/2>0$ and variance $\delta^2=((n-m+1)^2-1)/12\ge0$. This may be thought of in terms of \emph{rolling fair dice}. The number of sides of the die is equal to $n-m+1$ (its `size'). The distribution of random variable $K\sim\kappa_{m,n}$ is given by \[\P(K=k)=\frac{1}{n-m+1}\quad\text{for}\quad k\in\{m,m+1,\dotsb,n-1,n\}\] and the pgf is \begin{equation}\label{eq:pgf}\psi_{m,n}(t) = \frac{t^m-t^{n+1}}{(n-m+1)(1-t)}.\end{equation}

\subsection{Orthogonal dice}

In this section we introduce and characterize the orthogonal dice as those rectangular dice having mean-variance equality. Despite the statistical origin and nature of the objects as distributions, they surprisingly possess a number of number-theoretic properties with relevance to applied mathematics.  

First we define orthogonal dice and random measures. 

\begin{de}[Orthogonal die] We refer to $\kappa_{m,n}$ as an \emph{orthogonal die} if $N=(\kappa_{m,n},\nu)$ is orthogonal. In turn we call $N$ an orthogonal die random measure.\end{de}

For $m=0$ and arbitrary $f,g\in\mathscr{E}_{\ge0}$, we have covariance \[\Cov(Nf,Ng)=\frac{n}{2}\left(\nu(fg)+\frac{(n-4)}{6}\nu f\nu g\right)\] which is orthogonal for $n=4$. This is our first orthogonal die example.

\begin{prop}[Orthogonal die on $\{0,1,2,3,4\}$]$\kappa_{m,n}$ on $\{0,1,\dotsc,n-1,n\}$ is an orthogonal die for $n=4$.\end{prop} 

Note that we if change the support of $\kappa_{m,n}$ to $\{1,\dotsc,n\}$, then we have covariance \[\Cov(Nf,Ng)=\frac{n+1}{2}\left(\nu(fg)+\frac{(n-7)}{6}\nu f\nu g\right)\] so $N$ is orthogonal for $n=7$, i.e. seven-sided die. 

\begin{prop}[Orthogonal die on $\{1,2,3,4,5,6,7\}$]$\kappa_{m,n}$ on $\{1,2,\dotsc,n-1,n\}$ is an orthogonal die for $n=7$.
\end{prop}

We find there are infinitely many such orthogonal dice. This is the foundational theorem of the article. 

\begin{samepage}
\begin{thm}[Existence and uniqueness of orthogonal rectangular dice]\label{thm:1} Orthogonal dice $\kappa_{m,n}$ with support $\{m,m+1\dotsc,n-1,n\}$ for integers $0\le m<n$ are completely enumerated by the collection \begin{equation}\label{eq:K}\mathcal{S}=\{\kappa_{m,n}: (m,n)\in A\}\subset\mathcal{K}\end{equation} where \begin{equation}\label{eq:A}A=\{(m,n): k=1,2,4,5,7,8,\dotsb, m=(k^2-1)/3, n=2k+m+2\}\end{equation} with $|\mathcal{S}|=\infty$ and $m/n\rightarrow1$ as $m,n\rightarrow\infty$. Moreover, for each $\kappa_{m,n}\in\mathcal{S}$, the integer $n-m+1$ is a product of one or more primes each having value equal to or greater than five. \end{thm} 
\end{samepage}
\begin{proof} 
Recall that for arbitrary $f,g\in\mathscr{E}_{\ge0}$, we have covariance $\Cov(Nf,Ng)=c\nu(fg)+(\delta^2-c)\nu f\nu g$. Orthogonality requires $c=(m+n)/2=\delta^2=((n-m+1)^2-1)/12$, so we solve $\delta^2-c=0$ for $n$, giving solutions $n=m+2\pm 2 \sqrt{3 m+1}$, where $m/n\rightarrow1$ as $m,n\rightarrow\infty$. We take the positive root so that $m<n$ and denote the solution $n=h(m)$. $n$ is an integer whenever $3m+1=k^2$ for some integer $k\ge1$. Then this is equivalent to $m=(k^2-1)/3$ being an integer for integer $k$, which is the case for all integers $k\ge1$ except multiples of three. To see this, we recall that the integers $k^2$ and $k^2-1$ are coprime, so if the factorization of $k^2$ contains 3, then the factorization of $k^2-1$ does not; conversely, and relevant for our case, if the factorization of $k^2-1$ contains 3, then the factorization of $k^2$ does not. $k^2$ has the same prime factors as $k$, so the factorization of $k$ cannot contain three. Hence $k$ takes values of all positive integers except multiples of three, $k\in\N_{\ge1}\setminus 3\N_{\ge1}$. We put $n=h(m)=m+2+2 \sqrt{3 m+1}=m+2k+2$. Notice that $n-m+1=2k+3$ enumerates all odd integers starting with five except multiples of three. Therefore the factorization of $n-m+1$ may contain any prime except 2 or 3, i.e. the primes starting with 5. \end{proof}

The first 15 orthogonal dice are shown below in Table~\ref{tab:4}.

\begin{table}[h!]
\begin{center}
\begin{tabular}{cccc}
\toprule
$m$ & $n$ &$c$ & $n-m+1$ \\\midrule
0 &4 &2 &   5\\
1&7 & 4&   7\\
5&15 &10& 11\\ 
8&20 &14& 13\\
16&32 &24& 17\\
21&39 &30& 19\\
33&55 &44& 23\\
40&64 &52& 25\\
56&84 &70& 29\\
65&95 &80& 31\\
85&119 &102& 35\\
96&132 &114 & 37\\
120&160 &140& 41\\
133&175 &154& 43\\
161&207 &184& 47\\
\bottomrule
\end{tabular}\caption{First 15 orthogonal dice}\label{tab:4}
\end{center}
\end{table}

\FloatBarrier

We have the index or canonical parameter of the orthogonal dice. 
\begin{rem}[Index / canonical parameter] The set $\mathcal{S}$ is indexed by $\mathcal{I}=\N_{\ge1}\setminus 3\N_{\ge1}$. This is the canonical parameter of the orthogonal dice. The index $k\in\mathcal{I}$ has position $\lceil\frac{2}{3}k\rceil$.
\end{rem}

The most important object of the orthogonal dice is their Diophantine equation. Recall $A$ \eqref{eq:A} from Theorem~\ref{thm:1}. 

\begin{thm}[Diophantine]\label{thm:equation} \[m^2-2mn+n^2-8m-4n=0\quad\text{iff}\quad (m,n)\in A.\]\end{thm}

We give a key result that the union of the supports of the orthogonal dice forms the natural numbers and induces partitions of $\N_{\ge0}$ and $\N_{\ge1}$ respectively. This has natural applications in applied mathematics in the study of partitions and their properties. \begin{samepage}
\begin{cor}[Naturals]\label{cor:naturals}\[\bigcup_{\kappa_{m,n}\in\mathcal{S}}\{m,\dotsb,n\} = \bigcup_{\substack{\kappa_{m_k,n_k}\in\mathcal{S}:\\k=1+3j, j\ge0}}\{m_k,\dotsb,n_k\} =  \N_{\ge0}\] and \[\bigcup_{\kappa_{m_k,n_k}\in\mathcal{S}: k > 1}\{m_k,\dotsb,n_k\} = \bigcup_{\substack{\kappa_{m_k,n_k}\in\mathcal{S}:\\k=2+3j, j\ge0}}\{m_k,\dotsb,n_k\} =  \N_{\ge1}.\]\end{cor}
\end{samepage}
\begin{proof} The first result for $\N_{\ge0}$ follows from putting $m(k) = (k^2-1)/3$ and $n(k)=2k+m(k)+2$ and noting that $m(k+1)<m(k+2)<n(k)$ for all $k\ge1$, so that starting with $m(1)=0$ and $n(1)=4$, all the integers are enumerated at least once for the given sequence of $k$. The second result for $\N_{\ge0}$ follows from $m(k+3)=n(k)+1$ for all $k\ge1$. The results for $\N_{\ge1}$ similarly follow.
\end{proof} 

We give a few remarks about the orthogonal dice.

\begin{rem}[Poisson convergence]\label{re:plt}The family of discrete uniform support parameters $A$ satisfies the hypothesis $m/n\rightarrow1$ as $m,n\rightarrow\infty$ of Theorem~\ref{thm:main}, so by Theorem~\ref{thm:main2} ranging over this family and thinning retrieves the Poisson law. An illustration of this convergence is depicted in Figure~\ref{fig:convergence}, depicting the mass functions of Poisson and a sequence of $a$-thinned orthogonal dice with mean $114$.  

\begin{figure}[h]
\centering
\includegraphics[width=4.5in]{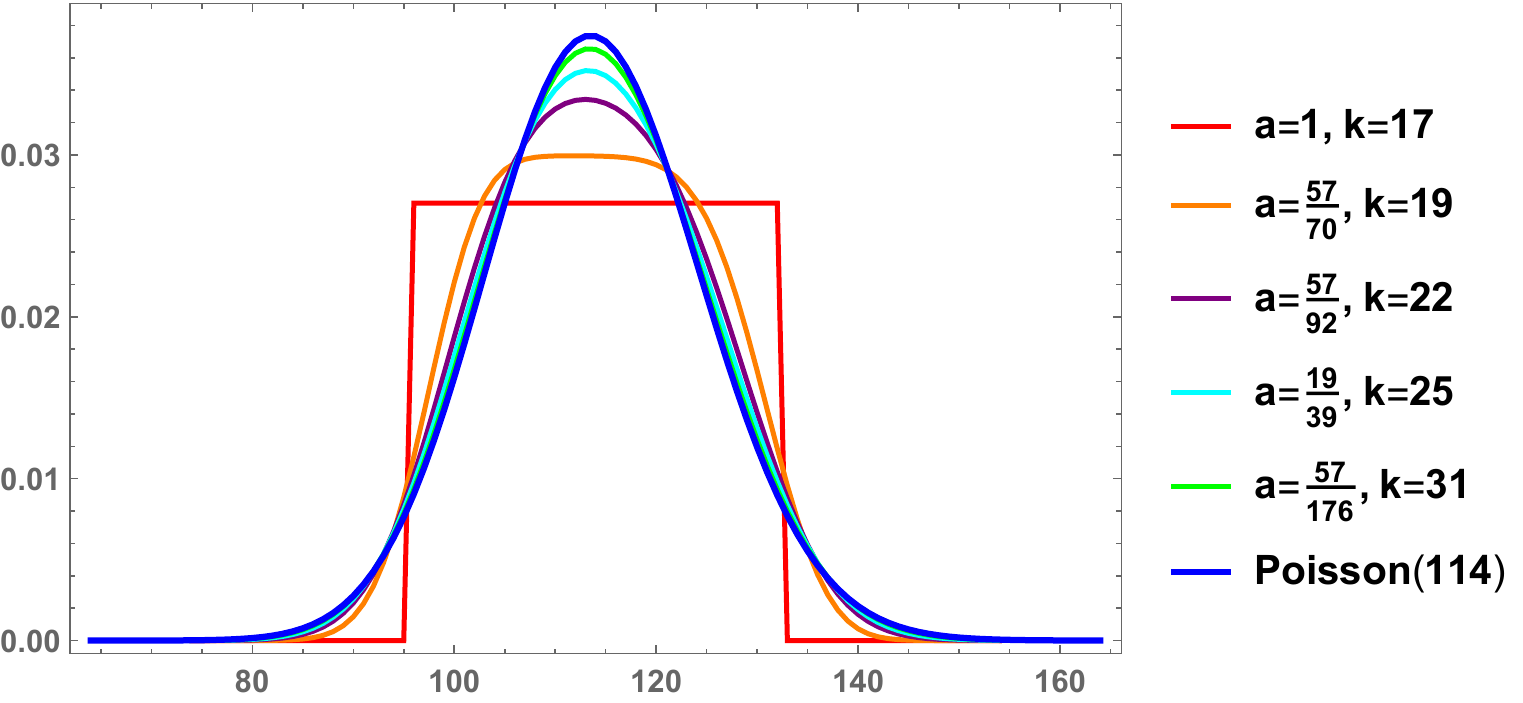}
\caption{Convergence: Mass functions of Poisson and a sequence of $a$-thinned orthogonal dice with mean $114$}\label{fig:convergence}
\end{figure}
\FloatBarrier
\end{rem}

\begin{rem}[Asymptotics]\label{re:asy}The number of sides $n-m+1=2k+3$ scales linearly $\Theta(k)$, whereas the mean $c=\frac{1}{3}(k+1)(k+2)$ scales quadratically $\Theta(k^2)$. Hence as $k$ gets large we have $n-m+1=2k+3\ll \frac{1}{3}(k+1)(k+2)=c$.
\end{rem}
\begin{rem}[Nearest die]\label{re:near} The mean is given by $c(k)=\frac{1}{3}(k+1)(k+2)$ for $k\in\mathcal{I}$. The orthogonal die $\kappa_{m,n}$ with mean closest to some given mean $c^*$ is located at index \[k^*=\argmin_{k\in\mathcal{I}}|c(k)-c^*|\approx\frac{1}{2}(\sqrt{1+12c^*}-3).\]
\end{rem}

\subsubsection{Statistical properties}
A key consequence of Theorem~\ref{thm:1} is a measurable disjoint partition of $\mathcal{K}$ into infinite orthogonal, positively correlated, and negatively correlated components. %This Corollary will be critical to future endeavors.

\begin{samepage}
\begin{cor}[Disjoint partition of $\mathcal{K}$]\label{cor:part} $\mathcal{K}$ \eqref{eq:U} may be partitioned into three disjoint infinite subsets \[\mathcal{K} = \mathcal{S}\cup\mathcal{C}_+\cup\mathcal{C}_-\] where $\mathcal{S}$ \eqref{eq:K} and \begin{align*}\mathcal{C}_+ &= \{\kappa_{m,n}: \text{integers }m\ge0, n>m+2+2\sqrt{3m+1}\}\\\mathcal{C}_- &= \{\kappa_{m,n}: \text{integers }m\ge0, 0<n<m+2+2\sqrt{3m+1}\}.\end{align*} The notation $\mathcal{C}_-$ and $\mathcal{C}_+$ indicates the sign of the covariance of the corresponding random measures $N=(\kappa_{m,n},\nu)$ for disjoint functions.  Elements of $\mathcal{C}_-$ are called negative dice, and elements of $\mathcal{C}_+$ are called positive dice. % for $\kappa_{m,n}\in\mathcal{C}_-$ and $\kappa_{m,n}\in\mathcal{C}_+$. 
\end{cor} 
\end{samepage}
\begin{proof} Apply Theorem~\ref{thm:1} to retrieve $\mathcal{S}$, where $\delta^2-c=0$ for $n=m+2+2\sqrt{3m+1}$ for the $m$ there. Then for every integer $m\ge0$ we have $\delta^2-c<0$ for integers $0<n<m+2+2\sqrt{3m+1}$, and $\delta^2-c>0$ for integers $n>m+2+2\sqrt{3m+1}$. \end{proof}

Now that we have identified $\mathcal{S}$, we give statistical properties of $\mathcal{S}$. Towards this, the following proposition gives the mean, variance, and covariance of orthogonal die random measures. This is restatement of the statistics of the mixed binomial process. Then we give a couple remarks.

\begin{samepage}
\begin{prop}[Campbell]\label{prop:orthogonaldie} Let $N=(\kappa_{m,n},\nu)$ be an orthogonal die random measure on $(E,\mathscr{E})$ and put $c=(m+n)/2$.  Then the Laplace functional of $N$ is \[L(f) = \E e^{-Nf} = \psi_{m,n}(\nu e^{-f})\for f\in\mathscr{E}_{\ge0},\] the mean and variance of $Nf$ (Campbell's formulas) for $f\in\mathscr{E}_{\ge0}$ are \begin{align}\E Nf &= c\nu f\label{eq:mu}\\\Var Nf &= c\nu f^2,\label{eq:var}\end{align} and the covariance of $Nf$ and $Ng$ for arbitrary $f,g\in\mathscr{E}_{\ge0}$ is \begin{equation}\label{eq:cov} \Cov(Nf,Ng)=c\nu(fg).\end{equation} 
\end{prop}\end{samepage}\begin{proof}These follow from the formulas of the mixed binomial process and the definition of orthogonality.\end{proof}
\begin{rem}[No distributional closure of dice]\label{re:thin} The law of the restriction is not uniform for $0<a<1$, i.e. uniform random measures are not closed under thinning.
\end{rem} 

\begin{rem}[Mass function of $N\ind{A}$]\label{eq:density} The mass function of $N\ind{A}$ is given by \[\P(N\ind{A}=k)=\psi_A^{(k)}(0) / k! =[t^k]\psi_A(t)= \frac{1}{2\pi i}\oint_{C}\frac{\psi_A(t)}{t^{k+1}}\D t\for k\in\N_{\ge0}\] where $\psi_A(t)=\psi_{m,n}(1-a+at)$ \eqref{eq:la1} with $\support(N\ind{A})=\{0,1,\dotsb,n\}$ for $A\subset E$ and $\support(N\ind{E})=\{m,\dotsb,n\}$. \end{rem}

We give a remark on quadratic integer sequences and a result on orthogonal triangular probability counting measures.  %, and a result on product random measures.

\begin{rem}[Quadratic sequences] The orthogonal dice support minimum sequence $m_k=(k^2-1)/3$ is quadratic. Other integer sequences with quadratic growth are the triangular numbers $T_k=k(k+1)/2$ with generating function $\varphi_T(t) = t/(1-t)^3$ and the summatory totient function $\Phi(k)\equiv\sum_{n\le k}\phi(n)\sim 3k^2/\pi^2$.  
\end{rem}

Let $\circledast$ denote the convolution operation of measures on $\N_{\ge0}$, i.e., $\nu\circledast\mu$ defined by \[(\nu\circledast\mu)f = \sum_{x\ge0}\nu\{x\}\sum_{y\ge0}\mu\{y\}f(x+y),\] naturally extending to $k$-fold convolutions. 

 \begin{prop}[Orthogonal triangular] Every orthogonal die law $\kappa_{m,n}\in\mathcal{S}$ with pgf $\psi_{m,n}$ is associated to an orthogonal triangular law $\triangle_{m,n}=\kappa_{m,n}\circledast\kappa_{m,n}$ with pgf $\psi_{m,n}^2$ and rectangular support $\{2m,\dotsb,2n\}$.\end{prop}
 
 More generally we have that the orthogonal dice generate orthogonal probability measures at all orders. Let $\kappa^k=\kappa\circledast\dotsb\circledast\kappa$ denote the $k$-fold convolution of the probability counting measure $\kappa$.
 \begin{thm}[Orthogonal sums]\label{thm:sums} Every orthogonal die $\kappa_{m,n}\in\mathcal{S}$ with pgf $\psi_{m,n}$ is associated through $k$-fold convolutions to an infinite family of orthogonal probability counting measures $\{\kappa^{k}_{m,n}: k\ge1\}$ with pgfs $\{\psi_{m,n}^k: k\ge1\}$ and rectangular supports $\{\{km,\dotsb,kn\}: k\ge1\}$.
 \end{thm}

\subsubsection{Applied mathematics properties}
Now we move into some applied mathematical properties of the orthogonal dice. 

First, the scalings of the sizes and supports of the orthogonal dice  are non-trivial. For example, they satisfy the conditions for the refined Gaussian conjecture of Montgomery and Soundararajan in number theory \citep{MS}, implying the number of primes in the parts of the partitions follows an under-dispersed Gaussian law, instead of a Poisson law.
\begin{prop}[Montgomery and Soundararajan Scaling]\label{prop:ms}
Let $(m_k,n_k)\in A$ for each $k\in \N_{\geq 1}\setminus 3\N_{\geq 1}$. Then, as $k\to+\infty$,
\[\frac{n_k-m_k+1}{\log(m_k)}\rightarrow+\infty\qquad\text{and}\qquad \frac{n_k-m_k+1}
{m_k}\rightarrow0.\]
\end{prop}

The size of a die at index $k\in\mathcal{I}$ is $l_k=n_k-m_k+1=2k+3$. Consider a process where, starting from the origin $(0,0)$, we cycle right-up-left-down with increments equal to $l_k$. The first coordinate has location $(0,l_1,l_1,l_1-l_4,l_1-l_4,l_1-l_4+l_7,\dotsb)=(0,5,5,-6,-6,11,11,\dotsb)$. The second coordinate has location $(0,0,l_2,l_2,l_2-l_5,l_2-l_5,l_2-l_5+l_8,l_2-l_5+l_8,\dotsb)=(0,0,7,7,-6,-6,13,13,\dotsb)$. Because the horizontal locations are assigned to the dice at odd positions and the vertical correspond to the dice at even positions, we obtain expressions for the horizontal and vertical locations \[\sum_{j\le k}(-1)^j(5+6j)=1+(-1)^k(4+3k),\quad \sum_{j\le k}(-1)^j(7+6j)=2+(-1)^k(5+3k).\] Hence the locations expand at the same linear rate in $k$. This process produces a `spiral' depicted below in Figure~\ref{fig:spiral} for the first 100 dice, where each successive point of the process is joined by a line. Because the indices are repeated and given the relationship between orthogonal die index in $k\in\mathcal{I}$ and position as $j=\lceil \frac{2}{3}k\rceil$, the spiral resides within square $[-3j/2,3j/2]^2$, evaluating to $[-150,150]^2$ for the first $j=100$ dice. The length of the line is $\frac{3}{2}j(2+j)\Rightarrow 15\,300$ for $j=100$. 

\begin{figure}[h]
\centering
\includegraphics[width=3.5in]{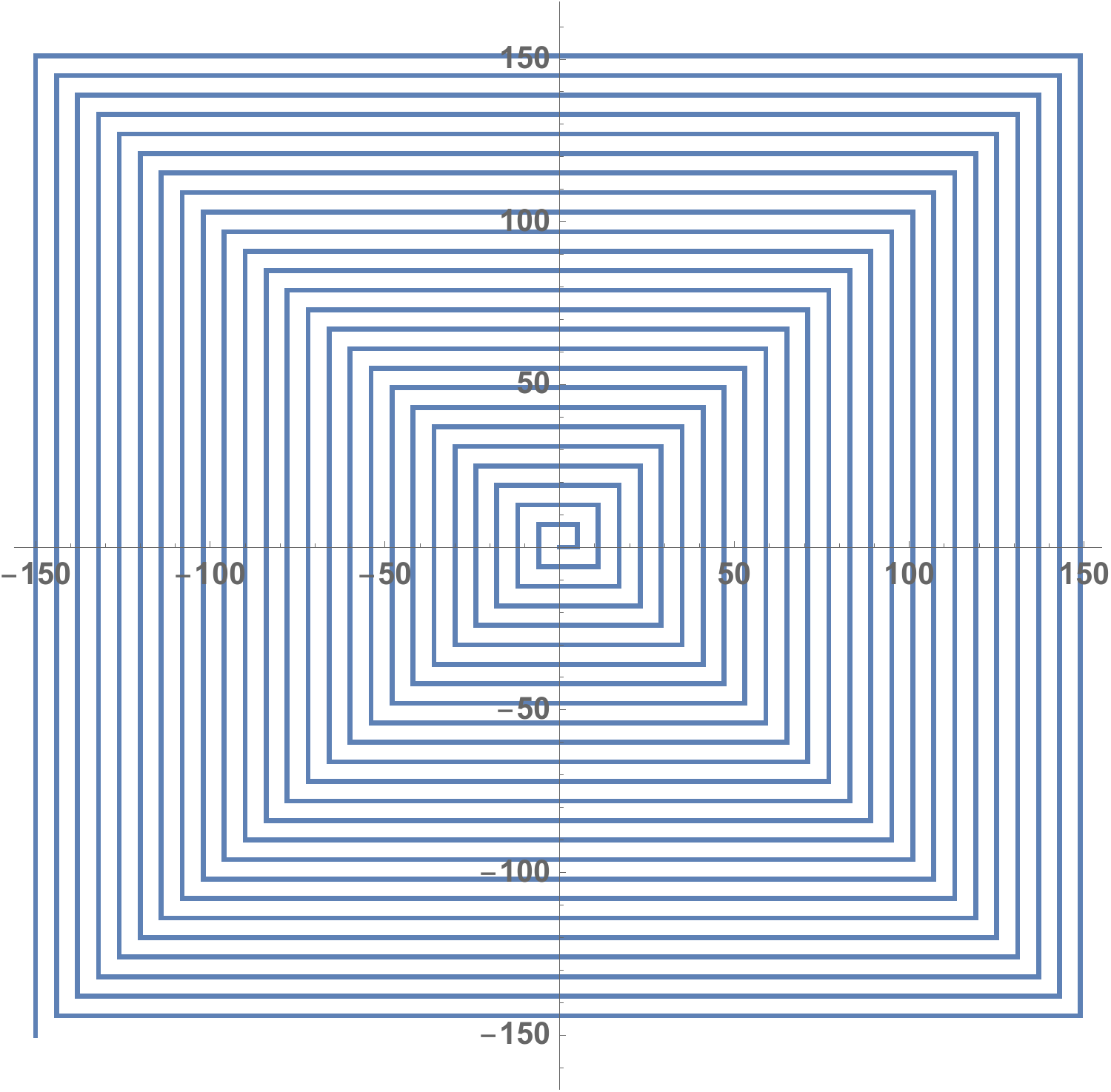}
\caption{`Spiral' produced by the first 100 orthogonal dice}\label{fig:spiral}
\end{figure}
\FloatBarrier

As a preliminary to establishing a key corollary of Theorem~\ref{thm:1} on co-prime arithmetic progressions, we give a result on coprimality. 

\begin{lem}[Coprime]\label{lem:coprime} \[\gcd(p,2p-1)=\gcd(p,2p+1)=1\for p\in\N_{\ge1}\] and \[\gcd(p,\tfrac{1}{2}p-1)=\gcd(p,\tfrac{1}{2}p+1)=1\for p\in4\N_{\ge1}.\]
\end{lem}\begin{proof}Recall that for coprime $a$ and $b$ such that $a+b=c$, then $\gcd(a,b,c)=\gcd(a,b)=\gcd(a,c)=\gcd(b,c)=1$. Then $\gcd(p,p\pm1)=1$ implies $\gcd(p,p+p\pm1)=1$. Putting $p=4x$ and $z=2x\pm1$, consecutive odd integers are coprime $\gcd(z,z\pm2)=1$ such that $\gcd(z+z\pm 2,z)=1$, giving $\gcd(p,\frac{1}{2}p\pm1)=\gcd(4x,2x\pm1)=\gcd(2x+1+2x-1,2x\pm1)=\gcd(z+z\pm 2,z)=1$.\end{proof}

Another key corollary of the main theorem is given, showing that the support of each orthogonal die may be partitioned using the information of the next orthogonal die, and the sizes of the partitions are coprime. This follows from the structure of their arithmetic progressions and Lemma~\ref{lem:coprime}. 
 
 \begin{samepage}
\begin{cor}[Coprimality]\label{cor:coprime}For $k\in\mathcal{I}$, let $m_k\equiv(k^2-1)/3$ and $n_k\equiv m_k+2k+2$, indicate $m_{k+1}$ as the support minimum of the next orthogonal die from $k$, and partition the support \begin{align*}\{m_k,\dotsb,n_k\} &=A_k\cup B_k\equiv \{m_k,\dotsb,m_{k+1}\}\cup\{m_{k+1}+1,\dotsb,n_k\} \\&=C_k\cup D_k\equiv \{m_k,\dotsb,m_{k+1}-1\}\cup\{m_{k+1},\dotsb,n_k\}.\end{align*} Then the following pairs of naturals are coprime \[(i)\quad |A_k| \text{ and }|B_k|,\quad\quad (ii)\quad |C_k|\text{ and }|D_k|,\quad\quad(iii)\quad|A_k|\text{ and }|C_k|,\quad\text{and}\quad(iv)\quad |B_k|\text{ and } |D_k|;\] the sequences $A=(|A_k|)$, $B=(|B_k|)$, $C=(|C_k|)$, and $D=(|D_k|)$ are non-repeating in $k$; and \[\{|C_1|,|B_2|,|C_4|,|B_5|,\dotsb,|C_k|,|B_{k+1}|,\dotsb\}=\N_{\ge1}.\] \end{cor}
\end{samepage}
\begin{proof} 
The third and fourth coprimality claims are trivially true. We focus on the first two. 

In the first, we have \begin{align*}|A_k| &= m_{k+1}-m_k+1 \\ |B_k| &= n_k-m_{k+1}. \end{align*} Take $k_1(j) = 1 + 3j$ and $k_2(j)=2+3j$ for $j\ge0$ such that \begin{align*}|A_{k_1(j)}| &= (k_2^2(j)-1)/3 - (k_1^2(j)-1)/3 + 1 =2+2j\\|B_{k_1(j)}|&= (k_1^2(j)-1)/3 + 2k_1(j) + 2 - (k_2^2(j)-1)/3 = 3+4j\end{align*} Then by Lemma~\ref{lem:coprime}  $\gcd(2+2j,3+4j)=\gcd(2+2j,2*(2+2j)-1)=1$ for $j\ge0$ so that the progressions are coprime. 

Now we do the other case \begin{align*}|A_{k_2(j)}| &= (k_1^2(j+1)-1)/3 - (k_2^2(j)-1)/3 + 1 =5+4j\\|B_{k_2(j)}|&= (k_2^2(j)-1)/3 + 2k_2(j) + 2 - (k_1^2(j+1)-1)/3 =2+2j \end{align*} Then $\gcd(5+4j,2+2j)=\gcd(2+2j,2*(2+2j)+1)=1$ for $j\ge0$ and the progressions are coprime.

Next we have for $C_k$ and $D_k$ sizes \begin{align*}|C_k| &= m_{k+1}-m_k \\ |D_k| &= n_k-m_{k+1}+1. \end{align*} In terms of $k_1$ and $k_2$ and $j$, these are $|C_{k_1(j)}| = 1+2j $ and $|D_{k_1(j)}| = 4+4j$. Then $\gcd(1+2j,4+4j)=\gcd(\frac{1}{2}(4+4j)-1,4+4j)=1$ for $j\ge0$ and the progressions are coprime. 

For the last case, we have $|C_{k_2(j)}| = 4+4j$ and $|D_{k_2(j)}| = 3+2j$. Then $\gcd(3+2j,4+4j)=\gcd(\frac{1}{2}(4+4j)+1,4+4j)=1$ for $j\ge0$ and the progessions are coprime.

The last claim follows as $\N_{\ge1}=\{1+2j: j\ge0\}\cup\{2+2j:j\ge0\}$.
\end{proof} 

In Table~\ref{tab:size} we show the values of the sizes of the partitions for the first 15 orthogonal dice.
\begin{table}[h!]
\begin{center}
\begin{tabular}{ccc|cc|cc}
\toprule
$m$ & $n$ & $n-m+1$ &$|A|$ & $|B|$ & $|C|$ & $|D|$ \\\midrule
0 &4 & 5&2 &3& 1& 4\\
1&7  & 7& 5 &2& 4& 3\\
5&15 & 11 &4 &7 &3 &8\\ 
8&20 & 13 &9 &4 &8 &5\\
16&32 & 17 &6 &11& 5& 12\\
21&39 & 19 &13& 6& 12& 7\\
33&55 & 23 &8 &15& 7& 16\\
40&64 & 25 &17 &8& 16& 9\\
56&84 & 29 &10 &19& 9& 20\\
65&95 & 31&21 &10& 20& 11\\
85&119 & 35 &12 &23& 11& 24\\
96&132 & 37 &25 &12 &24& 13\\
120&160 & 41&14 &27 &13 &28\\
133&175 & 43 &29 &14 &28& 15\\
161&207 & 47 &16 &31 &15 &32\\
\bottomrule
\end{tabular}\caption{Sizes of partitions of the first 15 orthogonal dice}\label{tab:size}
\end{center}
\end{table}

\FloatBarrier

We give a remark on the arithmetic progressions of the orthogonal dice.
\begin{samepage}\begin{rem}[Arithmetic progressions-Chebyshev's Bias] For $k\in\mathcal{I}$, let $m_k\equiv(k^2-1)/3$, $n_k\equiv m_k+2k+2$, and $l_k \equiv n_k - m_k + 1=2k+3$. Then $\{l_k: k\in\mathcal{I}\}$ is identified to arithmetic progressions $\{5+6j: j\ge0\}$ and $\{7 + 6j: j\ge0\}$. Similarly, the partitions yield arithmetic progressions. The progression for $|B_{k_1}|$ of $\{3+4j: j\ge0\}$ appears in the phenomenon called \emph{Chebyshev's Bias} \citep{chebyshev} where there are more primes yielded by $3+4j$ than $1+4j$ (up to some limit). 
\end{rem}\end{samepage}

The generating functions for the quantities in the corollary on coprimarily are given in the following proposition. 
% \begin{samepage}
\begin{prop}[Generating functions]\label{prop:gen} Letting $f_k=(m_k+n_k)/2$, $l_k=n_k-m_k+1$, $a_k=|A_k|$, $b_k=|B_k|$, $c_k=|C_k|$, and $d_k=|D_k|$, the generating functions of the component sequences of $\{(k,m_k,n_k,f_k,l_k,a_k,b_k,c_k,d_k): k\in\mathcal{I}\}$ are given by the rational functions \[\psi_k(t)=\frac{1+t+t^2}{(1-t)^2(1+t)},\quad \psi_m(t)=\frac{t +4t^2+t^3}{(1-t)^3(1+t)^2},\quad\psi_n(t)=\frac{4+3t-t^3}{(1-t)^3 (1+t)^2},\quad \psi_f(t) =\frac{2+2t+2t^2}{(1-t)^3 (1+t)^2}\]\[\psi_l(t) = \frac{5+2t-t^2}{(1-t)^2(1+t)},\quad\psi_a(t) = \frac{2+5t-t^3}{(1-t^2)^2},\quad\psi_b(t) = \frac{3+2t+t^2}{(1-t^2)^2},\quad\psi_c(t)=\frac{1+4t+t^2}{(1-t^2)^2} ,\quad\psi_d(t)=\frac{4+3t-t^3}{(1-t^2)^2}\] each encoding a linear recursive sequence with constant coefficients. 
\end{prop}
% \end{samepage}

\begin{rem}[Decompositions]
Furthermore, because of the interlaced nature of canonical index $k$, each generating function may be decomposed into a sum of two generating functions as well as a product of the generating function $1/(1-t)$, corresponding to the geometric series, and another generating function. This is a sum representation of the sequence.
\end{rem}

The decompositions are given in the following proposition.

\begin{prop}[Decomposition] Each of the nine generating functions may be decomposed into the sum and product of two generating functions, the sum corresponding to even-odd indices, per the following: 
\begin{align*}\psi_k(t)&=\frac{1+t+t^2}{(1-t)^2(1+t)}=\frac{1+2t^2}{(1-t^2)^2}+\frac{t(2+t^2)}{(1-t^2)^2}= \frac{1+t+t^2}{1-t^2}\frac{1}{1-t}\\
\psi_m(t)&=\frac{t+4t^2+t^3}{(1-t)^3(1+t)^2} = \frac{t^2(5+t^2)}{(1-t^2)^3} + \frac{5(1+5t^2)}{(1-t^2)^3} = \frac{t+4t^2+t^3}{(1-t^2)^2}\frac{1}{1-t}\\
\psi_n(t)&=\frac{4+3t-t^3}{(1-t)^3(1+t)^2} = \frac{(1+t^2)(4-t^2)}{(1-t^2)^3} + \frac{t(1-7t^2)}{(1-t^2)^3}=\frac{4+3t-t^3}{(1-t^2)^2}\frac{1}{1-t}\\
\psi_f(t)&=\frac{2+2t+2t^2}{(1-t)^3(1+t)^2} = \frac{2+4t^2}{(1-t^2)^3} + \frac{4t+2t^3}{(1-t^2)^3} = \frac{2+2t+2t^2}{(1-t^2)^2}\frac{1}{1-t}\\
\psi_l(t)&=\frac{5+2t-t^2}{(1-t)^2(1+t)} = \frac{5+t^2}{(1-t^2)^2} + \frac{t(7-t^2)}{(1-t^2)^2} = \frac{5+2t-t^2}{1-t^2}\frac{1}{1-t}\\
\psi_a(t)&=\frac{2+5t-t^3}{(1-t^2)^2}=\frac{2}{(1-t^2)^2} + \frac{t(5-t^2)}{(1-t^2)^2} = \frac{2+5t-t^3}{(1-t)(1+t)^2}\frac{1}{1-t}\\
\psi_b(t)&=\frac{3+2t+t^2}{(1-t^2)^2}=\frac{3+t^2}{(1-t^2)^2}+\frac{2t}{(1-t^2)^2} = \frac{3+2t+t^2}{(1-t)(1+t)^2}\frac{1}{1-t}\\
\psi_c(t)&=\frac{1+4t+t^2}{(1-t^2)^2}=\frac{1+t^2}{(1-t^2)^2} + \frac{4t}{(1-t^2)^2} = \frac{1+4t+t^2}{(1-t)(1+t)^2}\frac{1}{1-t}\\
\psi_d(t)&=\frac{4+3t-t^3}{(1-t^2)^2} = \frac{4}{(1-t^2)^2} + \frac{t(3-t^2)}{(1-t^2)^2} = \frac{4+3t-t^3}{(1-t)(1+t)^2}\frac{1}{1-t}.
\end{align*}

\end{prop}

Moreover, each generating function allows recovery of an explicit expression of the terms through the inverse $Z$-transform of the generating function. The inverse transforms as given below.

\begin{prop}[Impulse-response]\label{prop:impulsereponse} The inverse $Z$-transform of each of the nine generating functions of the orthogonal dice evaluated at $t=1/z$ where $z$ is the $Z$-transform variate gives recovery of the sequences on $\N_{\ge0}$ as \begin{align*}j&\in\N_{\ge0}\\
k(j)&=\frac{1}{4}(3+(-1)^j+6j)\\m(j)&=\frac{1}{8} \left(2 j \left(3 j+(-1)^j+3\right)+(-1)^j-1\right)\\n(j)&=\frac{1}{8} \left(6 j (j+5)+(-1)^j (2 j+5)+27\right)\\f(j)&=\frac{1}{8} \left(2 j \left(3 j+(-1)^j+9\right)+3 (-1)^j+13\right)\\l(j)&=\frac{1}{2} \left(6 j+(-1)^j+9\right)\\a(j)&=\frac{1}{2} \left(3 j+(-1)^{j+1} (j+1)+5\right)\\b(j)&=\frac{1}{2} \left((-1)^j+3\right) j+(-1)^j+2\\c(j)&=\frac{1}{2} \left(3+(-1)^{j+1}\right) (j+1)\\d(j)&=\frac{1}{2} \left((-1)^j+3\right) (j+2).\end{align*}

\end{prop}

We give some examples of the generating functions and their linear recursive sequences. 
\begin{rem}[Examples] Consider the generating function for the sequence of $l_k=n_k-m_k+1$ \[\psi_l(t) = \frac{5+2t-t^2}{1-t-t^2+t^3}.\] Then \[l_j = l_{j-1} + l_{j-2} - l_{j-3},\quad l_1 = 5,\quad l_2 = 7, \quad l_3 = 11.\] For $f_k=(m_k+n_k)/2$, we obtain \[f_j = f_{j-1}+2f_{j-2}-2f_{j-3}-f_{j-4}+f_{j-5},\quad f_1=2, f_2=4, f_3=10, f_4=14, f_5=24.\] Similarly, for $a_k=|A_k|$, we have \[\psi_a(t)=\frac{2+5t-t^3}{(1-t^2)^2}\] with \[a_j = 2a_{j-2} - a_{j-4},\quad a_1 = 2,\quad a_2 = 5,\quad a_3 = 4,\quad a_4 = 9.\]
\end{rem}

We give two sums of reciprocals of sequences---one transcendental, another infinite.  
\begin{rem}[Sums of reciprocals] We have \[\sum_{k\in\mathcal{I}}\frac{1}{f_k} =\sum_{k\in\mathcal{I}}\frac{3}{(k+1)(k+2)}= \frac{1}{3}(9-\sqrt{3}\pi)\simeq 1.1862 <\zeta(3) < \sqrt{2} < \varphi<\zeta(2)\] and \[\sum_{k\in\mathcal{I}}\frac{1}{l_k} = +\infty,\] where the first is transcendental, $\varphi$ is the golden ratio, and $\zeta$  is the Riemann zeta function. The transcendental number is obtained from a quadratic sequence. Other examples of quadratic sequences having transcendental sums of reciprocals are heptagonal numbers and square numbers. Note that \[\sum_{k\ge1}\frac{3}{(k+1)(k+2)}=\frac{3}{2},\] so that the transcendental property follows from the structure of the integrand and the canonical index set $\mathcal{I}$.
\end{rem}

The constant recursive sequences can be element-wise added or  multiplied or convolved, preserving constant recursion.

\begin{rem}[Operations] Given two sequences $a=(a_j)$ and $b=(b_j)$ with generating functions $F$ and $G$, we have generating function $H$ for the following operations: element wise sum and multiplication and the Cauchy product: \begin{align*}a+ b &\Leftrightarrow H(t)=F(t)+G(t)\\a\odot b&\Leftrightarrow H(t)=\frac{1}{2\pi}\int_{0}^{2\pi}F(\sqrt{t}e^{iz})G(\sqrt{t}e^{-iz})\D z\\a\circledast b&\Leftrightarrow H(t)=F(t)G(t).\end{align*}
\end{rem}

Each of the generating functions enables ready calculation of the discrete Fourier transform.

\begin{samepage}
\begin{rem}[Fourier transform] Let $(f(j))_{j=0}^\infty$ be the sequence of an orthogonal dice parameter given by Proposition~\ref{prop:impulsereponse}. Let $\psi_k(t)=\sum_{j=0}^{k-1}f(j)t^j$ be the generating function of the first $k$ terms of the sequence, $(f(j))_{j=0}^{k-1}$. Then $F(j)=\psi_k(\omega_j)$ where $\omega_j=\exp[2\pi\mathbf{i}j/k]$ is the $j$-the root of unity and $F(j)$ is the $j$-th component of the discrete Fourier transform of $(f(j))_{j=0}^{k-1}$, with $F=(F(j):j=0,1,\dotsb,k-1)$ the transformed sequence. The inverse Fourier transform is given by \[f(j) = (F^{-1}F)_j=\frac{1}{k}\sum_{z=0}^{k-1}\omega_j^{-z}F(z)\for j=0,1,\dotsb,k-1.\] Below in Figure~\ref{fig:fourier} the magnitude $|F|$ is plotted for the first 1000 points of the sequences of the canonical parameter $k$ and coprime parameter $|D|$. Notably, they are different: the canonical parameter has a pure decreasing spectrum, when in symmetry gives a bathtub, whereas the coprime parameter has a `well' appearance, which upon reflection gives a double-well shape. 

\begin{figure}[h!]
\centering
\begingroup
\captionsetup[subfigure]{width=3in,font=normalsize}
\subfloat[$k$]{\includegraphics[width=3in]{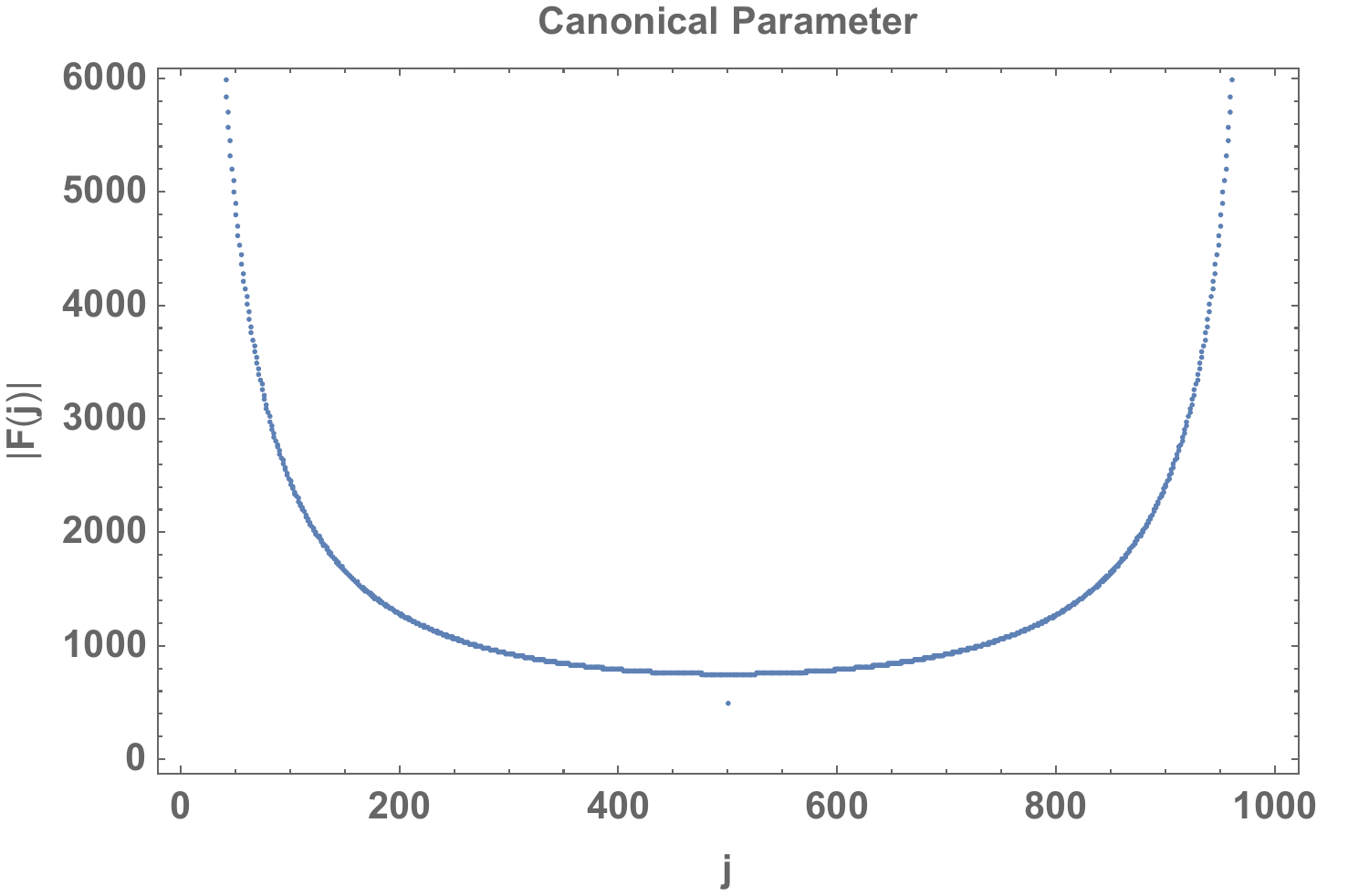}}
\subfloat[$|D|$]{\includegraphics[width=3in]{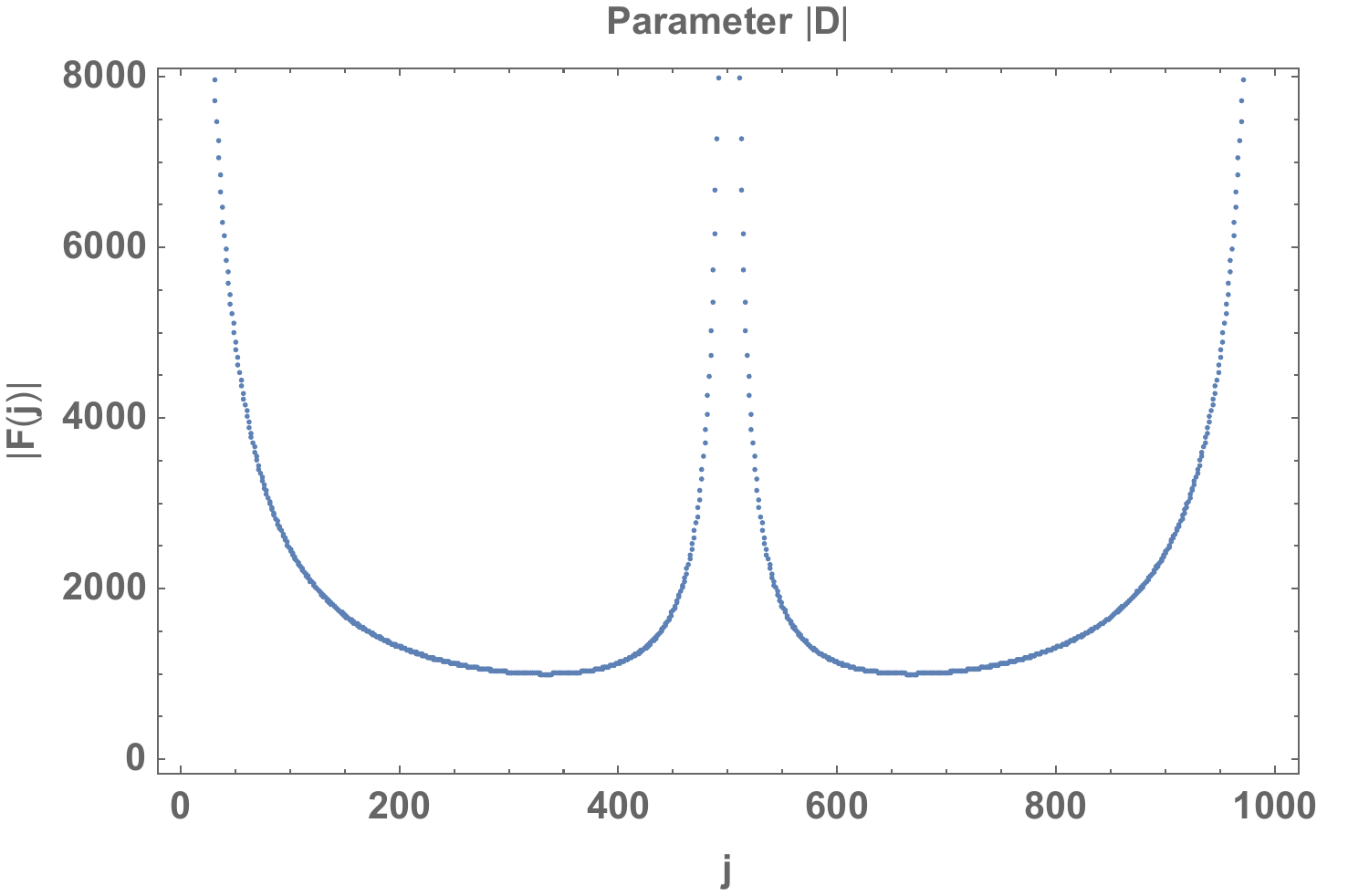}}\\
\endgroup
\caption{Plot of $|F|$ for canonical parameter $k$ and coprime parameter $|D|$}\label{fig:fourier}
\end{figure}\FloatBarrier
\end{rem}
\end{samepage}

Following from Theorem~\ref{thm:1} and Corollary~\ref{cor:coprime}, we have the following result.

\begin{prop}[Integer representation] Every natural number $n$ greater than or equal to five, coprime to two and to three, may be represented as a sum of coprime naturals $n=a+b$ such that $c=a-1$ and $d=b+1$ are coprime and $a\ne b\ne c\ne d$. %. Moreover, $c=a-1$ and $d=b+1$. 
\end{prop}

We give an immediate remark. 
\begin{rem}[Coprimality]\label{re:coprime} The proposition, in conjunction with the result that the sum of two coprime numbers is coprime to their summands, yields that $n$ is coprime to each of $\{a,b,c,d\}$. % and to their products $\{ab,cd,ac,bd\}$. 
\end{rem}

Corollary~\ref{cor:coprime} yields a co-prime splitting property of the orthogonal dice: any orthogonal die may be represented at least two ways as a mixture of non-orthogonal dice with coprime sizes. 

\begin{thm}[Co-prime splitting] Every orthogonal die $\kappa_{m,n}\in\mathcal{S}$ at index $k\in\mathcal{I}$ with size $n-m+1=2k+3$ may be represented as one of at least two mixtures of two non-orthogonal dice defined on sets $A_k$ and $B_k$---$\kappa_a=\text{Uniform}(A_k)$ and $\kappa_b=\text{Uniform}(B_k)$---as \[\kappa_{m,n}=\frac{|A_k|}{2k+3}\kappa_a+\frac{|B_k|}{2k+3}\kappa_b,\] such that $A_k$ and $B_k$ are rectangular and disjoint, partition the support $A_k\cup B_k=\{m,\dotsb,n\}$, and have coprime sizes $\gcd(|A_k|,|B_k|)=1$. 
\end{thm}

The following is another consequence of Theorem~\ref{thm:1}, providing a way of characterizing the construction of $\mathcal{S}$ in terms of prime numbers. 

\begin{samepage}
\begin{cor}[Prime construction]\label{thm:2} Let the integer $p$ be a product of one or more primes, where each prime is equal to or greater than five. Then there exists an orthogonal die $\kappa_{m,n}\in\mathcal{S}$ with $p$ sides, canonical parameter $k=(p-3)/2\in\mathcal{I}$, support parameters \begin{align*}m&=\frac{1}{12}(p-5)(p-1)\\n&=\frac{1}{12}(p+7)(p-1)\end{align*} mean \[c=\frac{1}{12}(p^2-1)\]  and position $\lceil\frac{2}{3}k\rceil=\lceil\frac{1}{3}p\rceil-1$.
\end{cor}\end{samepage}
\begin{proof} Theorem~\ref{thm:1} yields $\mathcal{S}$. As stated there, $p=n-m+1$, and this is unique for each orthogonal die (the orthogonal die are indexed by products of powers of primes greater than or equal to five). Thus for each such $p$, there exists an orthogonal die $\kappa_{m,n}$. To get the parameters $m$ and $n$, we set-up and solve the system of equations based on $c=\delta^2$ for $m$ and $n$: \begin{align*}\frac{m+n}{2}&=\frac{1}{12}(p^2-1)\\p&=n-m+1\end{align*} Set $n=p+m-1$ from the second equation and plug into the LHS of the first equation and solve for $m$. Similarly, set $m=n-p+1$ and solve for $n$. The index follows from $p=n-m+1=2k+3$.
\end{proof} 

We give a couple remarks on prime dice.

\begin{rem}[Prime dice] Let $p\ge5$ be prime. Then by Corollary~\ref{thm:2} each such prime number is uniquely identified to an orthogonal die. These are the prime dice. If these primes are taken from some family, then we carry the name over to the corresponding family of orthogonal dice, i.e. Mersenne primes and their orthogonal dice.
\end{rem}

\begin{rem}[Largest known prime die]\label{re:largest} $p=2^{82\,589\,933}-1$ is the largest known prime with $24\,862\,048$ digits and is Mersenne. Then by Corollary~\ref{thm:2} the corresponding orthogonal die $\kappa_{m,n}$ has $p$ sides,  parameters \begin{align*}m&=\frac{1}{12}(2^{82\,589\,933}-6)(2^{82\,589\,933}-2)\\n&=\frac{1}{12}(2^{82\,589\,933}+6)(2^{82\,589\,933}-2)\end{align*} and mean \[c=\frac{1}{12}2^{82\,589\,933}(2^{82\,589\,933}-2).\] It has canonical parameter \[k=2^{82\,589\,932}-2\in\mathcal{I}\] at position $\lceil\frac{2}{3}k\rceil=\lceil\frac{1}{3}(2^{82\,589\,933}-4)\rceil$.\end{rem}

Last, but not least, we see that the orthogonal dice are self-similar, where an orthogonal die of composite size may be decomposed into orthogonal dice of coprime sizes.

\begin{samepage}
\begin{thm}[Self-similarity]\label{thm:selfsimilar} Every independent roll $K$ of an orthogonal die $\kappa_{m,n}\in\mathcal{S}$ with coprime factorization $n-m+1=b_1\dotsb b_l$ is equivalent to $l$ independent rolls of orthogonal dice $\{K_i\sim\kappa_{m_i,n_i}\in\mathcal{S}: i=1,\dotsb,l\}$ at canonical indices $\{k_i=(b_i-3)/2\in\mathcal{I}: i=1,\dotsb,l\}$ with pairwise coprime sizes $b_1,\dotsb,b_l$. %, that is, $\psi_{m,n}\Leftrightarrow\prod_{i}^k\psi_{m_i,n_i}$ where $k_i=(b_i-3)/2\in\mathcal{I}$ for $i=1,\dotsb,l$
\end{thm}
\end{samepage}
\begin{proof} (Sufficiency) Given $n-m+1=b_1\dotsb b_l$ where $b_1,\dotsb,b_l\ge5$ are pairwise coprime, let $k_i=(b_i-3)/2\in\mathcal{I}$ and $m_i = (k_i^2-1)/3$ for $i=1,\dotsb,l$. Then $K\Rightarrow(K_1,\dotsb,K_l)$ follows from linear congruence relations $K_i=f_i(K)=(K-m\mmod b_i)+m_i$ for $i=1,\dotsb,l$. Independence follows from that of the congruences for coprime bases and uniformity follows from $\kappa_{m,n}\circ f_i^{-1}=\kappa_{m_i,n_i}$, belonging to $\mathcal{S}$ by Theorem~\ref{thm:1} at index $k_i=(b_i-3)/2$ in $\mathcal{I}$, for $i=1,\dotsb,l$. (Necessity) Note that $k=(b_1\dotsb b_l -3)/2\in\mathcal{I}$ following from pairwise coprime $b_1,\dotsb,b_l\ge5$ with $m=(k^2-1)/3$, $n=m+2k+2$, and $n-m+1=b_1\dotsb b_l$. Then $K\Leftarrow(K_1,\dotsb,K_l)$ follows from the Chinese remainder theorem as $K=\text{ChineseRemainder}(K_1-m_1,\dotsb,K_l-m_l | b_1,\dotsb,b_l)+m\in\{m,\dotsb,n\}$ for the given $(K_1,\dotsb,K_l)$ and $(b_1,\dotsb,b_l)$ such that $K\sim\kappa_{m,n}\in\mathcal{S}$ at index $k=(b_1\dotsb b_l-3)/2\in\mathcal{I}$.
\end{proof}

 \subsection{Generalizations} We discuss some generalizations of orthogonal die random measures: general orthogonal random measures,  superposition random measures, and product random measures. All orthogonal constructions are Poisson-like, holding for Poisson random measures. Note that Poisson random measures are additive (independent), stronger than orthogonality.
 
 The following gives a construction of general orthogonal random measures.
 
 \begin{samepage}
 \begin{thm}[Orthogonal random measure]\label{thm:orthogonalrm} Let $N=(\kappa_{m,n},\nu)$ be an orthogonal die random measure on $(E\times\R_{\ge0},\mathscr{E}\otimes\mathscr{B}_{\R_{\ge0}})$. Define \[L(A) = \int_{A\times\R_{\ge0}}N(\D x,\D z) z\for A\in\mathscr{E}.\] Then, $L$ is an orthogonal random measure on $(E,\mathscr{E})$ with Laplace functional \[\E e^{-Lf} = \psi_{m,n}(\int_{E\times\R_{\ge0}}\nu(\D x, \D z)e^{-f(x)z})\for f\in\mathscr{E}_{\ge0}.\] 
 \end{thm}\end{samepage}\begin{proof} By Fubini's theorem, $L$ is a random measure. $L(A)$ is determined by the trace of $N$ on $A\times\R_{\ge0}$. For finitely many disjoint sets $A,\dotsb,B$ in $\mathscr{E}$, the traces of $N$ over $A\times\R_{\ge0},\dotsb,B\times\R_{\ge0}$ are orthogonal by Proposition~\ref{prop:orthogonaldie}, so $L(A),\dotsb,L(B)$ are decorrelated and $L$ is orthogonal. The Laplace functional follows from definition. 
 \end{proof}
 
The statistics readily follow and are given below.
 
 \begin{rem}[Statistics] The mean and covariance are \[\E Lf = \frac{m+n}{2}\int_{E\times\R_{\ge0}}\nu(\D x,\D z)f(x)z,\quad \Cov(Lf,Lg) = \frac{m+n}{2}\int_{E\times\R_{\ge0}}\nu(\D x,\D z)f(x)g(x)z^2.\]
 \end{rem}
 
 An important concept is a superposition of random measures. The following result is for superpositions of orthogonal die random measures. They admit Poisson and Gaussian regimes.
\begin{samepage}
\begin{thm}[Superposition]\label{thm:super} Let $F$ be countable, e.g., $F=\{1,2,\dotsb\}$, and put $\mathscr{F}=2^F$. Let $\{N_j=(\kappa_{m_j,n_j},\nu_j): j\in F\}$ be an independency of orthogonal die random measures on $(E,\mathscr{E})$. Then the $\Sigma$-finite random measure $M$ on $(E\times F,\mathscr{E}\otimes\mathscr{F})$ defined by \[M(A\times\{j\})=N_j(A)\for A\in\mathscr{E},\quad j\in F\] is orthogonal and has $\Sigma$-finite mean measure $\E M = \lambda$ \[\E Mf = \lambda f = \int_{E\times F}c_y\nu_y(\D x)\delta_y(\D y)f(x,y)\for f\in(\mathscr{E}\otimes\mathscr{F})_{\ge0},\] Laplace functional \[L(f) = \E e^{-Mf} = \prod_{j\in F}\psi_{m_j,n_j}(\nu_j e^{-f(\cdot,j)})\sim \exp_-\lambda(1-e^{-f})\for f\in(\mathscr{E}\otimes\mathscr{F})_{\ge0},\] and covariance \[\Cov(Mf,Mg) = \lambda(fg)\for f,g\in(\mathscr{E}\otimes\mathscr{F})_{\ge0}.\]
\end{thm}\end{samepage}
\begin{proof}Orthogonality and the Laplace functional, mean measure, and covariance follow from definitions. The Poisson regime is implied by the convergence in distribution to Gaussian by central limit theorem and the convergence in distribution of Poisson to Gaussian: for increasing sequence $(f_n)\nearrow f$, \[N_n=\sum_{j\in F}N_jf_n(\cdot,j)\stackrel{D}{\rightarrow}\text{Gaussian}(\lambda f,\lambda f^2)\stackrel{D}{\leftarrow}\text{Poisson}(\lambda f_n)\] such that \[\frac{\prod_{j\in F}\psi_{m_j,n_j}(\nu_j e^{-f_n(\cdot,j)})}{ \exp_-\lambda(1-e^{-f_n})}\rightarrow 1.\] \end{proof}

An example of a superposition random measure is logarithmic. We give the construction and remark on Gaussian and Riemann zeta function connections. 

\begin{rem}[Logarithmic]\label{re:logarithmic} Take $(E\times F,\mathscr{E}\otimes\mathscr{F})=([0,1]\times\N_{\ge0},\mathscr{B}_{[0,1]}\otimes 2^{\N_{\ge0}})$ and $\nu_j(\D x) = \log(1/x)^{j}/j!\,\D x$ with mean $2^{-(j+1)}$ and $\sum_j 2^{-(j+1)}=1$. Consider the superposition $M$ of the orthogonal dice $N_j=(\kappa_{m,n},\nu_j)$ where $c=(m+n)/2$. Let $g\in\mathscr{E}_{\ge0}$ and $f(x,y)=g(x)\ind{F}(y)$ and note that $\sum_j\nu_j(\D x) = 1/x\, \D x$. Then $Mf$ has mean \[\E M f = \lambda f = c\int_{E\times F}\nu_y(\D x)\delta_y(\D y)g(x)\ind{F}(y) = c\int_E\frac{1}{x}g(x)\D x\for g\in\mathscr{E}_{\ge0}.\] When $g(x)=1$, then $\E Mf=+\infty$. For $g(x)=x$, the Laplace transform is given by \[\varphi(\alpha)=\E e^{-\alpha Mf} \sim\exp_-c(\gamma+\Gamma(0,\alpha)+\log(\alpha)) \for \alpha\in\R_{\ge0},\] $Mf\simeq\text{Gaussian}(c,c/2)$, and the statistics are \[\E Mf = c,\quad \Var Mf = c/2.\] 
\end{rem}

\begin{rem}[Gaussian]
To see the Gaussian distribution, conduct a second-order Taylor expansion of $\Gamma(0,\alpha)$ about zero, i.e., $\Gamma(0,\alpha)= -\gamma-\log(\alpha)+\alpha-\alpha^2/4+ O(\alpha^3)$ and invert the Fourier transform $\varphi(-i\alpha)=\exp[ic\alpha - \frac{1}{2}\frac{c}{2}\alpha^2]$, which yields the claimed Gaussian distribution $\P(Mf\in\D x)\simeq\frac{1}{\sqrt{c\pi}}e^{-(c-x)^2/c}\D x$. This is also directly obtained by the $\text{Gaussian}(\lambda f,\lambda f^2)$ limit.
\end{rem}

\begin{rem}[Zeta] Consider $\nu_j(\D x)=\log(1/x)^{j}/j!\,\D x$ and let $g(x)=1/(1-x)$. Then \[\E N_jg = c\int_{[0,1]}\nu_j(\D x)\frac{1}{1-x} = c\zeta(j+1)\for j=1,2,\dotsb\] where $\nu_j(g)=\zeta(j+1)=\E N_jg/c$ and $\zeta$ is the Riemann zeta function. Their superposition has infinite mean $\E Mf = c\int_{[0,1]}1/xg(x)\D x=+\infty$. 
\end{rem}

A generalization of an die random measure is to consider its product. We give the mean below.

\begin{samepage}
\begin{thm}[Product random measure]\label{thm:product} Consider the die random measure $N=(\kappa_{m,n},\nu)$ on $(E,\mathscr{E})$ where $\kappa_{m,n}\in\mathcal{K}$ with mean $c=(n+m)/2$ and variance $\delta^2=((n-m+1)^2-1)/12$. Form the product random measure $M=N\times N$ on $(E\times E,\mathscr{E}\otimes\mathscr{E})$ as \begin{equation}\label{eq:product}Mf = \int_{E\times E}N(\D x)N(\D y)f(x,y) \for f\in(\mathscr{E}\otimes\mathscr{E})_{\ge0}.\end{equation}Then, $M$ has mean \[\E M f = c(\nu\times I)f+(c^2+\delta^2-c)(\nu\times\nu)f\for f\in(\mathscr{E}\otimes\mathscr{E})_{\ge0}\] where $I$ is the identity kernel $I(x,\cdot)=\delta_x(\cdot)$.
\end{thm}\end{samepage}
\begin{proof} The result follows from using the covariance of the mixed binomial process with $f=\ind{A\times B}$ and applying a monotone class argument. Explicitly, \[\E Mf = c\int_E\nu(\D x)f(x,x) + (c^2+\delta^2-c)\int_{E\times E}\nu(\D x)\nu(\D y)f(x,y).\] 
\end{proof}

The product random measure of an orthogonal die has the same mean measure as of Poisson. 
\begin{cor}[Product random measure]\label{cor:product} Consider the orthogonal die random measure $N=(\kappa_{m,n},\nu)$ on $(E,\mathscr{E})$ and put $\lambda = \frac{m+n}{2}\nu$. Form the product random measure $M=N\times N$ on $(E\times E,\mathscr{E}\otimes\mathscr{E})$ as in \eqref{eq:product}. Then, $M$ has mean \[\E M f = (\lambda\times I)f+(\lambda\times\lambda)f\for f\in(\mathscr{E}\otimes\mathscr{E})_{\ge0}\] where $I$ is the identity kernel $I(x,\cdot)=\delta_x(\cdot)$. 
\end{cor}

\section{Applications}\label{sec:applications} We give some applications of the  orthogonal dice. First we show application to modeling bounding counts, a fundamental problem in statistics. There we see a correspondence between elements of dice and general probability counting measures through mean and variance and observe this yields mean-covariance equivalency of their mixed binomial processes. In the next application, we develop statistical inference algorithms for modeling of count data using the dice and their convolutions, including maximum likelihood and Bayesian estimation procedures. We show that count data is well modeled by convolutions of dice and efficiently estimated using the Metropolis-Hastings algorithm. In the third application, we use the orthogonal dice to construct canonical stochastic processes of Wiener and geometric Brownian motion. This result is important for manifold reasons, owing to the foundational nature of these processes. In the penultimate application, we discuss a statistical model for traffic flows based dice. There we construct mixed binomial processes representing systems of independently moving vehicles whose number of vehicles is bounded and uniform, with insightful covariance. In the last application we demonstrate application of the self-similarity property of orthogonal dice. We illustrate self-similarity with a simple example and show the property yields sparse uniform lattices of naturals, which uses square-root number of points to uniformly sample a hypercube of some volume of points. 

\subsection{Modeling systems having bounded counts} Suppose we have a collection of independent identically distributed count data $\mathbf{K}=\{K_i: i=1,\dotsb,z\}$, where $K_i\sim\kappa$ for probability counting measure $\kappa$ having finite mean $c$ and variance $\delta^2$. If we additionally know the mean and variance of the counting law $\kappa$, we can solve for the support parameters $m$ and $n$ of the corresponding die in terms of $c$ and $\delta^2$, obtaining (before rounding) \begin{align*} m&=\frac{1}{2} \left(2 c-\sqrt{12 \delta^2+1}+1\right)\\n&=\frac{1}{2} \left(2 c+\sqrt{12 \delta^2+1}-1\right).\end{align*} 

The $\kappa_{m,n}$ representation of $\kappa$ is the maximum entropy distribution on a contiguous interval of naturals that matches the first and second-order statistics of $\kappa$. The simple relation $(c,\delta^2)\Leftrightarrow(m,n)$ yields a surrogate die $\kappa_{m,n}$ for any given $\kappa$ having finite mean and variance. Moreover, we immediately know the classification of the die from $\sign(\delta^2-c)$, yielding ($-1$) negative, ($0$) orthogonal, and ($+1$) positive dice respectively. Moreover, letting $\nu$ be a deterministic probability measure on $(E,\mathscr{E})$, then the random counting measure $N_{m,n}=(\kappa_{m,n},\nu)$ has the same mean-covariance structure as the random counting measure $N=(\kappa,\nu)$, i.e., \[\E N_{m,n}f = \E Nf,\quad \Cov(N_{m,n}f,N_{m,n}g)=\Cov(Nf,Ng).\]

Recall the maximum entropy distribution on $\{m,\dotsb,n\}$ is $\kappa_{m,n}=\text{Uniform}\{m,\dotsb,n\}$ with pgf $\psi_{m,n}$. The distribution is symmetric, constant, and lacks a mode. We can construct new distributions by convolving dice. Towards this, consider independent $K_1,K_2\sim\kappa_{m,n}$ and define $K=K_1+K_2$. Then $\support(K)=\{2m,\dotsb,2n\}$ with size $2n-2m+1$ and $K$ follows the triangular distribution, with mode at $m+n$. That is, $K\sim\triangle_{m,n}=\kappa_{m,n}\circledast\kappa_{m,n}$ with pgf $\psi_{m,n}^2$. The mean and variance are respectively $c=\E K = m+n$ and $\delta^2=\Var K = ((n-m+1)^2-1)/6$. Here the $(c,\delta^2)\Leftrightarrow(m,n)$ solutions are \begin{align*} m&=\frac{1}{2} \left( c-\sqrt{6 \delta^2+1}+1\right)\\n&=\frac{1}{2} \left(c+\sqrt{6 \delta^2+1}-1\right).\end{align*} Again, the value of $\sign(\delta^2-c)$ determines the dispersion of the triangular law $\triangle_{m,n}$ and its generating die $\kappa_{m,n}$. The size of the support of a triangular law sharing the same statistics as a die is larger the die. Basically, the reduction in entropy enables more points to be added to the support for the same fixed mean and variance. Now consider $l$-fold convolutions of the $K_1,\dotsb,K_l\sim\kappa_{m,n}$, $K=K_1+\dotsb+K_l$. Then $K$ has pgf $\psi_{m,n}^l$, support $\support(K)=\{lm,\dotsb,ln\}$, mode $l(m+n)/2$, mean $c=l(m+n)/2$, and variance $\delta^2=l((n-m+1)^2-1)/12$. The $(c,\delta^2)\Leftrightarrow(m,n)$ correspondence is \begin{align*} m&=\frac{2c+l-\sqrt{l}\sqrt{l+12\delta^2}}{2 l}\\n&=\frac{2c-l+\sqrt{l} \sqrt{l+12 \delta^2}}{2 l}.\end{align*} 

A notable family of distributions with bounded counts is $\kappa=\text{Binomial}(l,p)$, where $K\sim\kappa$ has support $\support(K)=\{0,1,\dotsb,l\}$ with pgf $\psi(t)=(1-p+pt)^l$. Consider the negative die $\kappa_{0,1}=\text{Bernoulli}(1/2)$ with associated pgf $\psi_{0,1}=\frac{1}{2}(1-t)$. Then the $l$-order convolution is binomial $\kappa_{0,1}^{\circledast l}=\text{Binomial}(l,1/2)$, with pgf $\psi_{0,1}^l(t)=(\frac{1}{2}(1-t))^l$. Then we can $a$-thin the Bernoulli dice, giving pgf $\psi_a(t)=\psi(1-a+at)=\frac{1}{2}(1-(1-a+at))=1-\frac{a}{2}+\frac{a}{2}t$ and, applied to the convolution, pgf $\psi_a^l(t)=(1-\frac{a}{2}+\frac{a}{2}t)^l$, yielding a $\text{Binomial}(l,\frac{a}{2})$ distributed random variable $K$. Thus the family of binomial distributions for $0<p\le 1/2$ is recapitulated in terms of the maximum entropy Bernoulli die $\kappa_{0,1}$ and thinning. Finally, sending $l\rightarrow+\infty$ and $p\rightarrow0$ while fixing $np=c\in(0,\infty)$, a $\text{Poisson}(c)$ random variable $K$ is obtained. 

\subsection{Statistical inference of bounded counts}
Again consider the setting of having collection $\mathbf{K}=\{K_i:i=1,\dotsb,z\}$ distributed according to (here unknown) $\kappa$. We approximate $\kappa$ by the $l$-fold convolution of $\kappa_{m,n}$, denoted $\kappa_{m,n}^{\circledast l}$. The pgf is $\psi_{m,n}^l$ such that \[\P(K=k)=[t^k]\psi^l_{m,n}(t)=\frac{1}{2\pi i}\oint_C\frac{\psi_{m,n}^l(t)}{t^{k+1}}\D t\for k\in\{lm,\dotsb,ln\}.\] This is stably obtained by computing the residue of $\psi_{m,n}^l(t)/t^{k+1}$ at zero. Then the likelihood of $(m,n,l)$ given $\mathbf{K}$ is \[L(m,n,l|\mathbf{K})=\prod_{i=1}^z[t^{K_i}]\psi_{m,n}^l(t).\] The likelihood is zero if any of the samples fall outside of the bound, $|\mathbf{K}\cap\{lm,\dotsb,ln\}|<z$. Putting $a=\min\mathbf{K}$ and $b=\max\mathbf{K}$, we seek maximum-likelihood estimate (MLE) \[(\hat{m},\hat{n},\hat{l})=\argmax_{(m,n,l): lm\le a, ln\ge b}L(m,n,l|\mathbf{K}),\] obtaining the estimated law $\kappa_{\hat{m},\hat{n}}^{\circledast \hat{l}}$ with support $\{\hat{l}\hat{m},\dotsb,\hat{l}\hat{n}\}$, mean $\hat{c}=(\hat{l}\hat{m}+\hat{l}\hat{n})/2$, and $\hat{\delta}^2=\hat{l}((\hat{n}-\hat{m}-1)^2-1)/12$. Again the mean-covariance structure of the law can be classified according to value of $\sign(\hat{\delta}^2-\hat{c})$.

A Bayesian approach can be employed for estimation of $(m,n,l)$. First we assume a prior probability $\P(m,n,l)$. Then the posterior distribution of $(m,n,l)$ is given by \[\P(m,m,l|\mathbf{X})=\frac{L(m,n,l|\mathbf{X})\P(m,n,l)}{\P(\mathbf{X})}.\] The normalizing constant is difficult to compute, so a natural strategy is to employ a Monte Carlo method, such as Metropolis-Hastings (MH) \citep{metropolis1,hastings}, to obtain posterior sample estimates. Then the maximum a-posterior (MAP) estimate be obtained from the generated sequence. To use MH, we require a proposal density (conditional distribution) $\P(m',n',l'|m,n,l)$. There is a hierarchy in the parameters, where $l$ is at the base, and $(m,n)$ follow. Thus we factor $\P(m',n',l'|m,n,l)=\P(m',n'|m,n,l,l')\P(l'|l)$. For example, define \[\P(l'|l)=\text{Uniform}_{l'}\{l-1,l,l+1\}.\] Now put $\theta=(m,n,l,l',f)$ and $\theta'=(m',n',l',l,f)$ and define the integers residing in a circle centered about the $l/l'$-rescaled $(m,n)$ as \[\Theta'=\left\{(c,d)\in\N_{\ge0}^2: c<d, \sqrt{(c-\frac{l}{l'}m)^2+(d-\frac{l}{l'}n)^2}\le f\right\}\] and \[\Theta=\left\{(c,d)\in\N_{\ge0}^2: c<d, \sqrt{(c-\frac{l'}{l}m')^2+(d-\frac{l'}{l}n')^2}\le f\right\}.\] Note that $|\Theta|\approx|\Theta'|\approx\pi f^2$. Now consider uniform proposal density \[\P(m',n'|m,n,l,l',f)=\frac{1}{|\Theta'|}\approx\frac{1}{\pi f^2}\] where $f$ is the uniform search distance on the points. The size of $\Theta$ is not necessarily the same as $\Theta'$ and hence the proposal density is not symmetric. Therefore $|\Theta|$ and $|\Theta'|$ are used each iteration of the sampler. The acceptance probability is given by \begin{align*}\P((m,n,l)\rightarrow(m',n',l'))&=\min\{1,\frac{L(m',n',l'|\mathbf{X})\P(m',n',l')\P(m,n,l|m',n',l')}{L(m,n,l|\mathbf{X})\P(m,n,l)\P(m',n',l'|m,n,l)}\}\\&=\min\{1,\frac{L(m',n',l'|\mathbf{X})\P(m',n',l')|\Theta'|}{L(m,n,l|\mathbf{X})\P(m,n,l)|\Theta|}\}.
\end{align*} Let $\Upsilon$ be the space of possible parameter values, i.e., \[\Upsilon = \{(m,n,l)\in\N_{\ge0}\times\N_{\ge1}\times\N_{\ge1}: m<n, lm\le a, ln\ge b\}. \] The set $\Upsilon$ can be partitioned based on parameter regimes of dispersion into three disjoint sets $\Upsilon=\Upsilon_-\cup\Upsilon_0\cup\Upsilon_+$ corresponding to under-, equi-, and over-dispersed parameter sets. Therefore, a prior can be imposed by setting $\P(m,n,l)=\ind{\Upsilon_a}\circ(m,n,l)$ for respective $a\in\{-,0,+\}$. This essentially uses mean and variance information to impose constraints on the sampling.  

To illustrate the sampling method, we sample $z=50$ Poisson variates $\kappa=\text{Poisson}(c=114)$. The sample statistics are $(\min,\max,\mu,\sigma^2)=(93,135,\frac{5783}{50},\frac{210261}{2450})\approx(93,135,115.7,85.2)$.  Consequently, the sample is under-dispersed. The Poisson log-likelihood is $-183.518$. Now we utilize the sampling scheme. We assume a uniform prior on dispersion, i.e., $\P(m,n,l)=\ind{\Upsilon}\circ(m,n,l)$. We take $f=3$. We initialize the sampler to $(m,n,l)=(93,135,1)$ and run for 200 iterations. The maximum a-posterior sample corresponds to $(m,n,l)=(30,47,3)$, corresponding to $\kappa_{m,n}^{\circledast l}$ supported on $\{90,91,\dotsb,140,141\}$, with log-likelihood $-181.674$, mean $c=231/2=115.5$ and variance $\delta^2=323/4=80.75$. Posterior quantities are shown below in Figure~\ref{fig:mh}. The traces of $(m,n,l)$ are shown in Figures~\ref{fig:mh}(a) and (b), showing good exploration of the space. The log-likelhood is shown in Figure~\ref{fig:mh}(c), where the sampled models generally exhibit higher likelihood than the true Poisson law. The densities of the true Poisson law, a smooth kernel histogram estimate, and the dice MAP law are shown below in Figure~\ref{fig:mh}(d). For reference, the histogram estimate has log-likelihood $-178.181$. These results highlight the utility of the representation of $\kappa$ by $\kappa_{m,n}^{\circledast l}$: by using distributions having bounded support, higher quality fits to finite data are observed in comparison to using exact laws defined on infinite intervals. By construction, the histogram estimate is sculpted to the data, having the highest likelihood. The dice represents an intermediate between the (empirical) smooth kernel histogram and the Poisson law. 

\begin{figure}[h!]
\centering
\begingroup
\captionsetup[subfigure]{width=3in,font=normalsize}
\subfloat[$(m,n)$]{\includegraphics[width=3in]{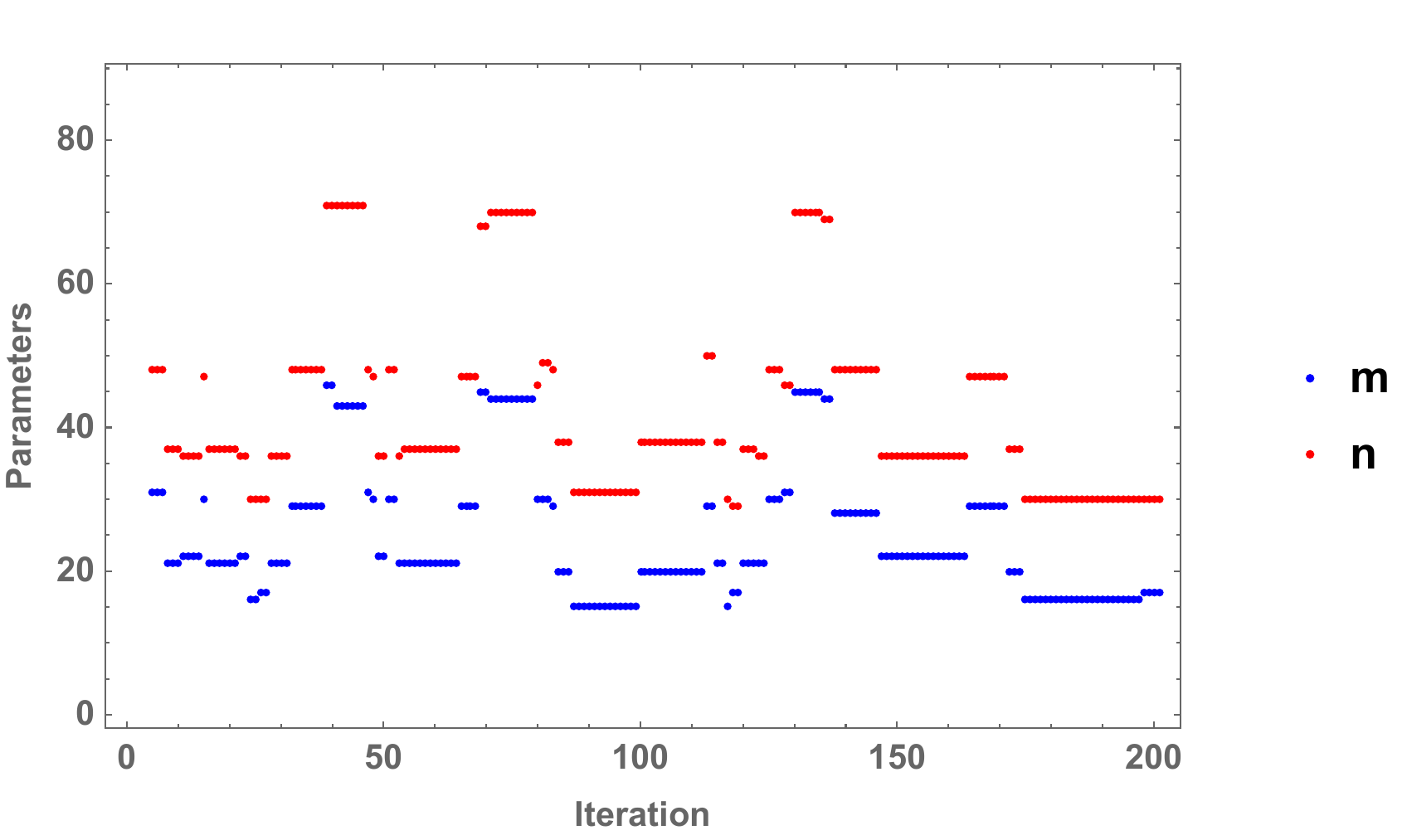}}
\subfloat[$l$]{\includegraphics[width=2.5in]{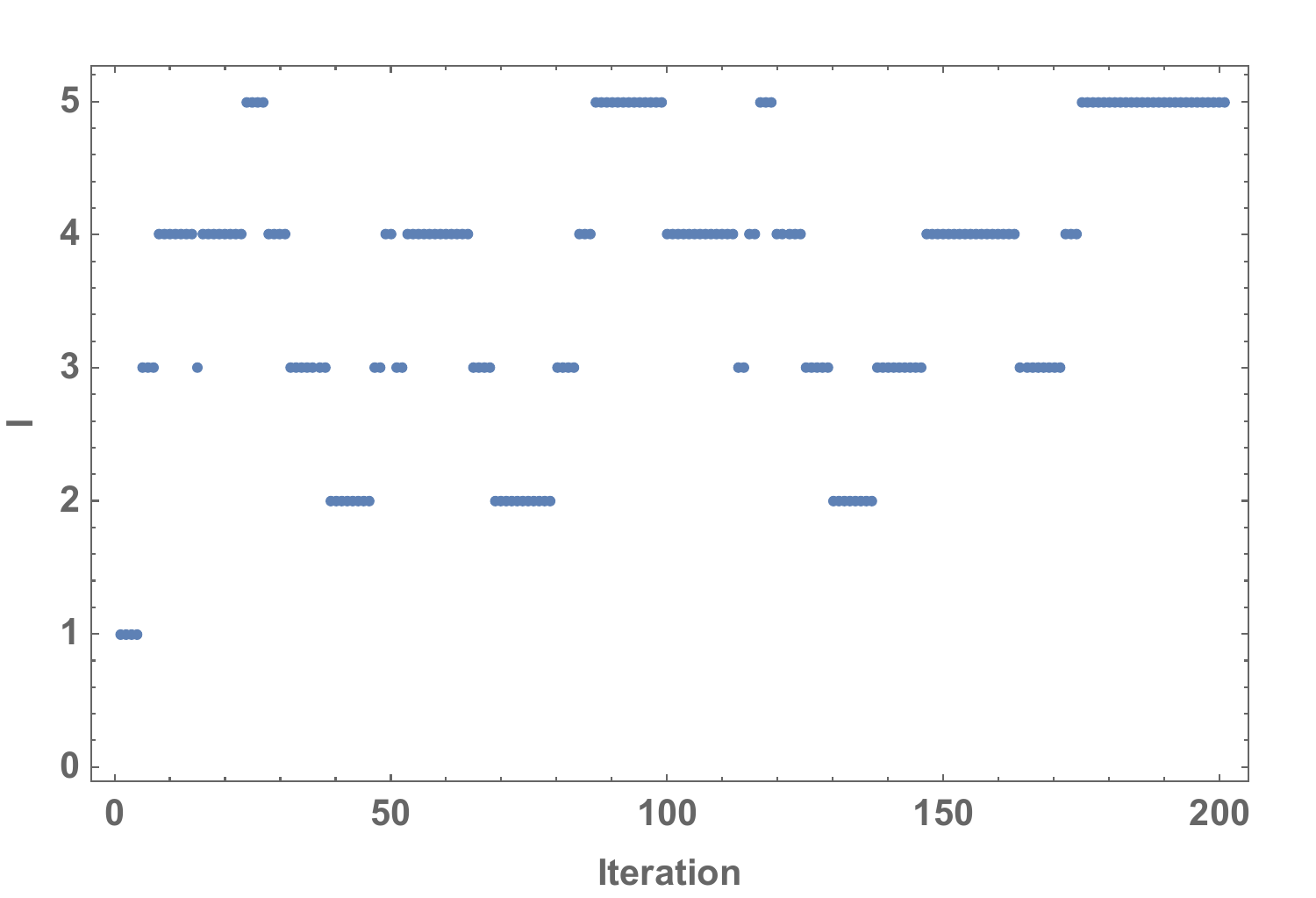}}\\
\subfloat[Log-likelihood]{\includegraphics[width=3in]{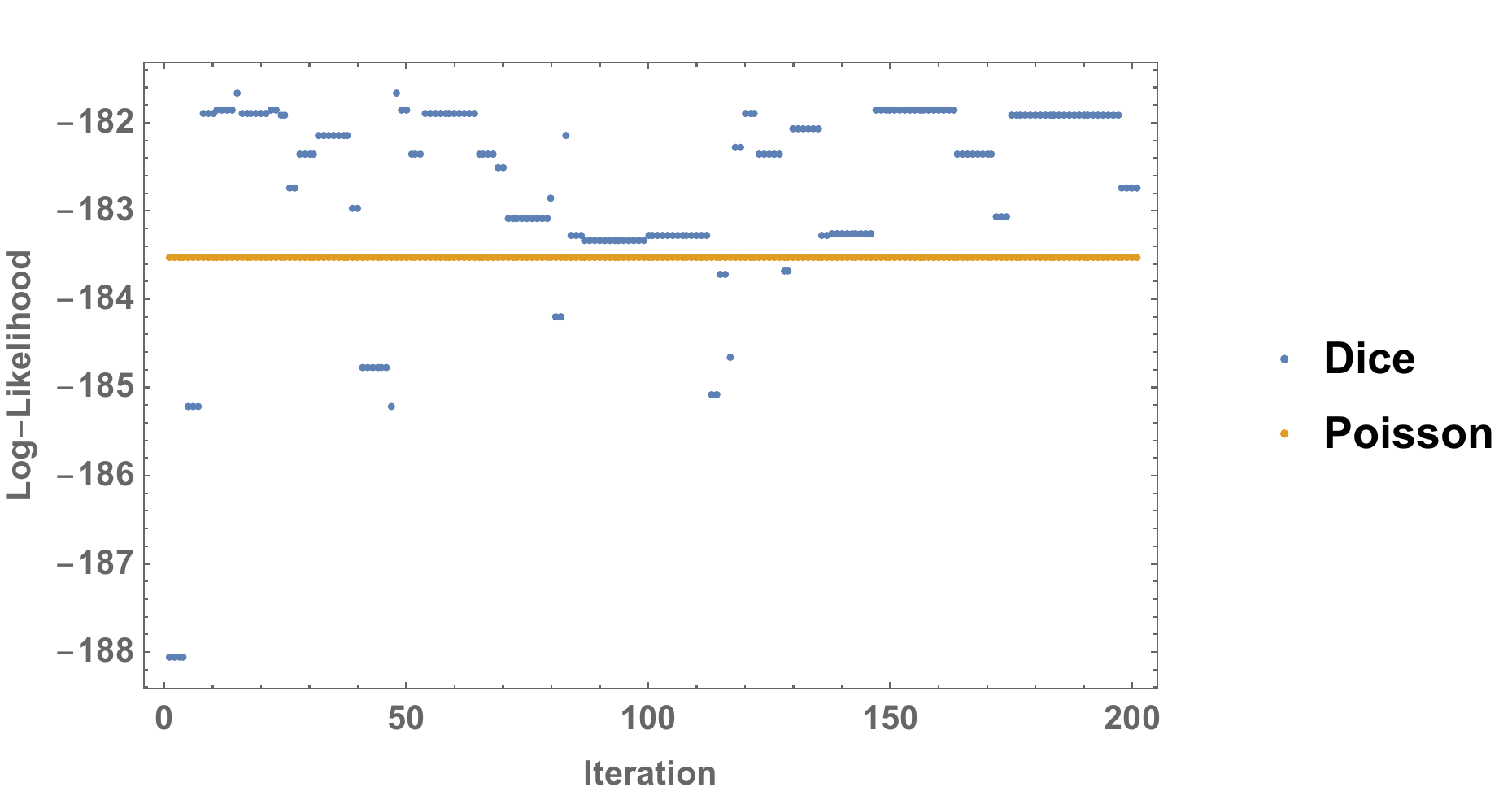}}\subfloat[Density]{\includegraphics[width=3in]{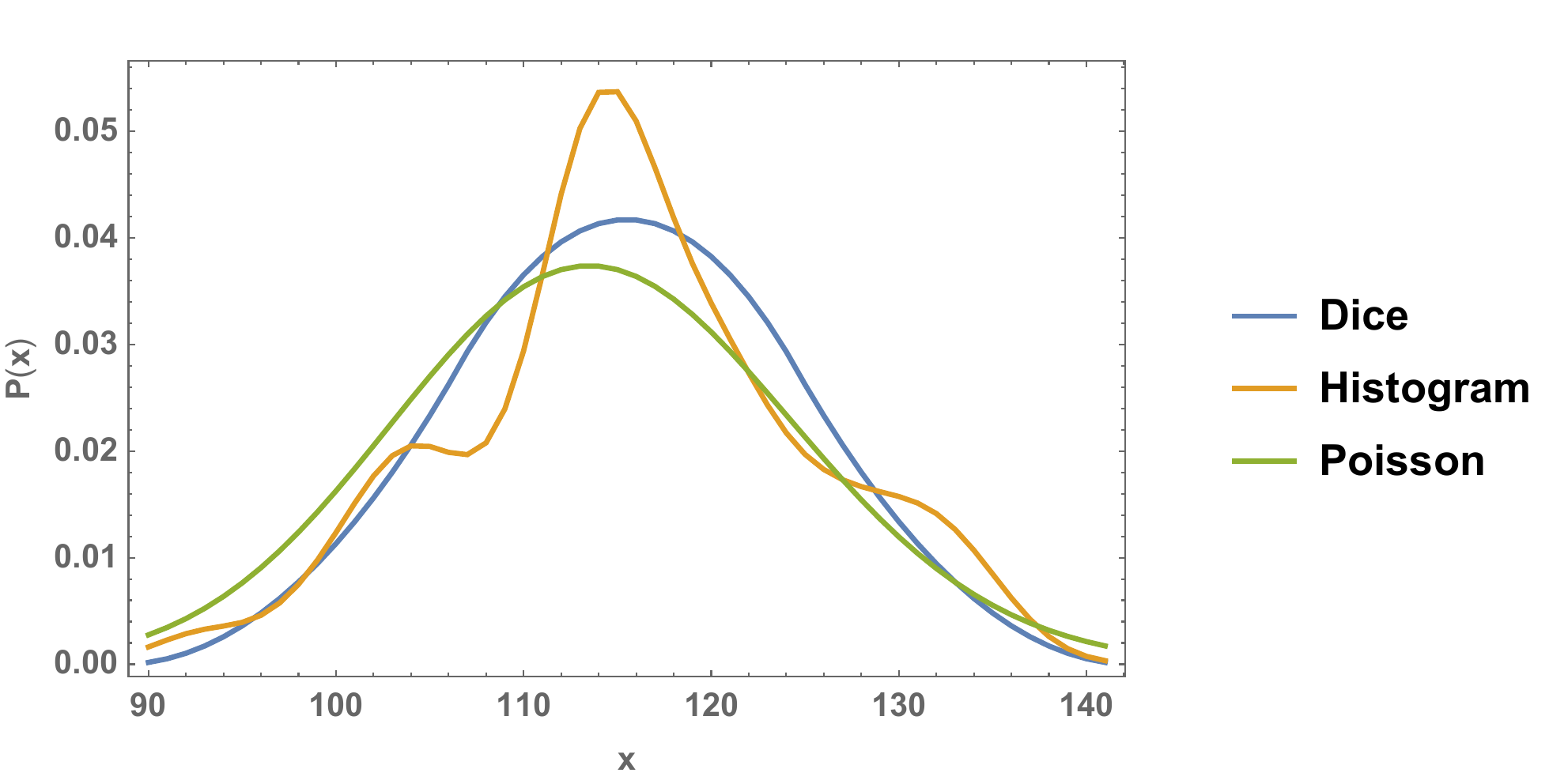}}\\
\endgroup
\caption{Trace and density plots from Metropolis-Hastings sampler}\label{fig:mh}
\end{figure}

\FloatBarrier

More generally, here for convolution order $l$ fixed, we can consider convolutions of dice $\kappa_{m_1,n_1},\dotsb,\kappa_{m_l,n_l}$ for support sequences $m=(m_1,\dotsb,m_l)$ and $n=(n_1,\dotsb,n_l)$ as \[\kappa_{m,n}=\kappa_{m_1,n_1}\circledast\dotsb\circledast\kappa_{m_l,n_l}.\] Then $K\sim\kappa_{m,n}$ has pgf $\psi_{m,n}(t)=\psi_{m_1,n_1}(t)\dotsb\psi_{m_l,n_l}(t)$. Defining $|m|=m_1+\dotsb+m_l$ and $|n|=n_1+\dotsb+n_l$, we obtain support $\{|m|,\dotsb,|n|\}$. Thus, the likelihood becomes \[L(m,n|\mathbf{K})=\prod_{i=1}^z[t^{K_i}]\psi_{m,n}(t).\] This gives MLE \[(\hat{m},\hat{n})=\argmax_{(m,n): |m|\le a, |n|\ge b}L(m,n|\mathbf{K}).\]

\subsection{Stochastic processes}
 
 In the following results, we show construction of canonical stochastic processes---Wiener and geometric Brownian motion---from the orthogonal dice. 
 
 \begin{samepage}
 \begin{thm}[Wiener Process] \label{thm:wiener} Let $N_k=(\kappa_{m_k,n_k},\nu_k)$ be an orthogonal die random measure on $(E,\mathscr{E})=(\R_{\ge0},\mathscr{B}_{\R_{\ge0}})$ for $k\in\mathcal{I}=\N_{\ge1}\setminus 3\N_{\ge1}$ with $c_k=(m_k+n_k)/2$ and $\nu_k=\text{Uniform}[0,T_k=n_k-m_k+1]$ and let $N_k(t) = N_k([0,t])$. Then $N_k(t)$ has mean and covariance \[\E N_k(t) = \frac{c_k}{T_k}t,\quad \C(N_k(s),N_k(t)) =\frac{c_k}{T_k}(s\wedge t)\] and $M_k=(M_k(t))$ in Skorokhod space $\mathcal{D}[0,T_k]$ defined by $M_k(t)=(N_k(t)-c_kt/T_k)/\sqrt{c_k/T_k}$ is a locally $L^2$-bounded martingale with quadratic variation $\langle M_k\rangle_t =\E M_k^2(t)=t$ such that $M_k=(M_k(t))$ converges in distribution in $\mathcal{D}[0,\infty)$ to the Wiener process $W=(W_t)$ as $k\rightarrow\infty$. 
 \end{thm}
 \end{samepage}
 \begin{proof} The first claim follows from the statistics of $N$, c.f., Proposition~\ref{prop:orthogonaldie}.  For the second, the strategy is to show that $M_t$ is a locally $L^2$-bounded martingale with quadratic variation $\langle M\rangle_t = t$ and apply Rebolledo’s theorem \citep{rebolledo} to yield weak convergence in Skorokhod space (see the appendix for adapted version of Rebolledo's theorem). Note that  \[\E N_t^2 = \frac{ct}{T} + \frac{c^2t^2}{T^2}\] such that \begin{align*}\E M_t^2 &= \E(N_t - ct/T)^2/(c/T) \\&= \E(N_t^2 - 2\frac{ct}{T}N_t + \frac{c^2t^2}{T^2})/(c/T) \\&=(ct/T + c^2t^2/T^2 - 2c^2t^2/T^2+c^2t^2/T^2)/(c/T)=t.\end{align*} Thus $M$ is locally $L^2$ bounded martingale. We have that \begin{align*}[M]_t\equiv [ M,M]_t &= \sum_{s\le t}(\Delta M_s)^2=\sum_{s\le t}(\sqrt{\frac{T}{c}}\Delta N_s)^2=\frac{T}{c}[N,N]_t = \frac{T}{c}N_t \end{align*} with compensator (predictable quadratic variation) $\langle M\rangle_t \equiv\langle M,M\rangle_t = t =\E M_t^2$. The jumps of $M$ are of size $\sqrt{T/c}$, which goes to zero as $m,n\rightarrow\infty$ so that the jumps disappear in the limit (recall quadratic $m_k=(k^2-1)/3$ and $n_k=m_k+2k+2$ and linear $T_k=n_k-m_k-1=2k+3$ for $k\in\N_{\ge1}\setminus 3\N_{\ge1}$). Thus Rebolledo's theorem applies and $M=(M_t)$ converges in distribution in $\mathcal{D}[0,\infty)$ as $k\rightarrow\infty$ ($m,n\rightarrow\infty$) to the Wiener process $W=(W_t)$. \end{proof}

 In Figure~\ref{fig:wiener} we show the Wiener construction for the $j=1000$ orthogonal die located at $k=1499$, yielding $M_t$ on time-set $[0,3001]$.  
\begin{figure}[h]
\centering
\includegraphics[width=4.5in]{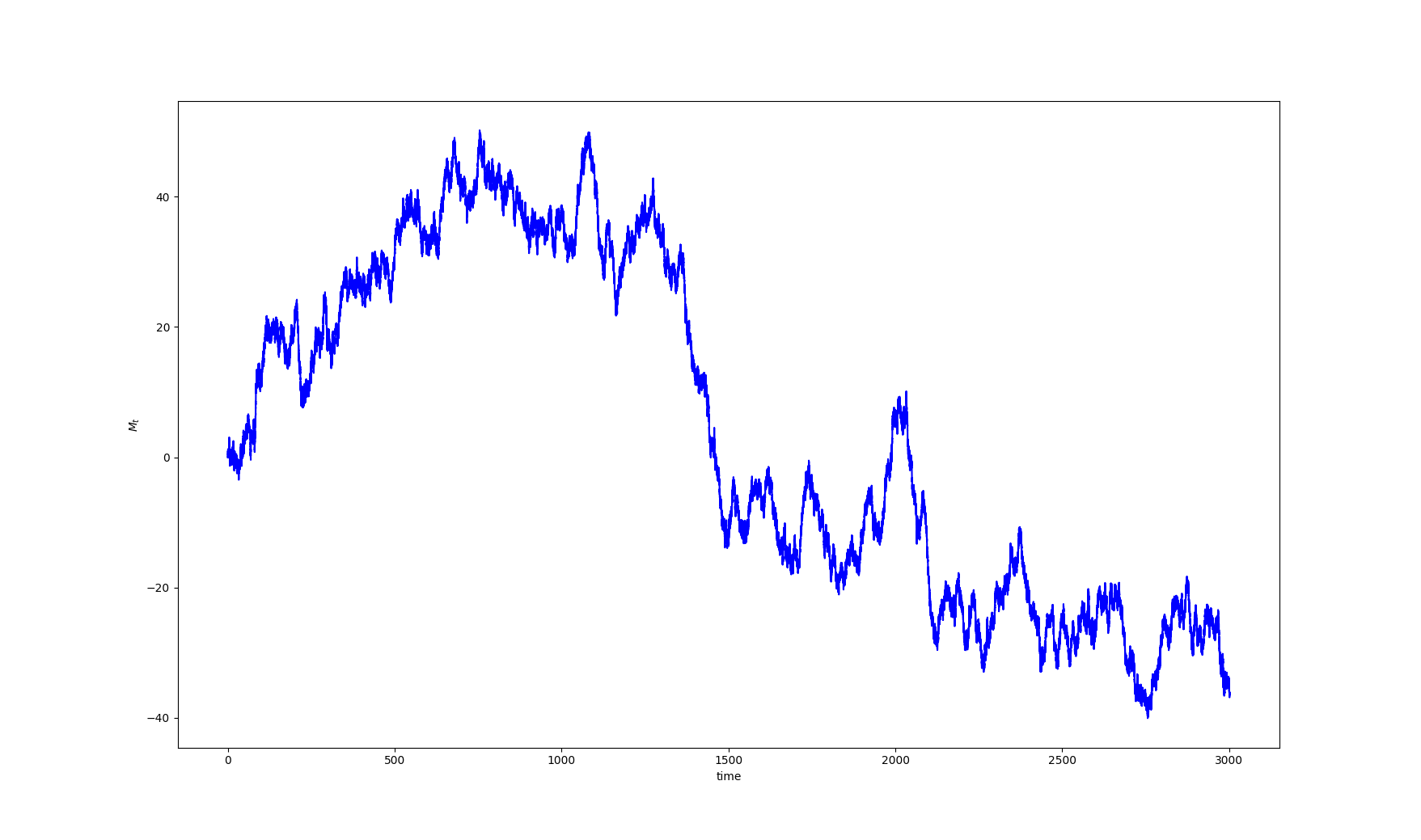}
\caption{Convergence: Orthogonal die construction of the Wiener Process}\label{fig:wiener}
\end{figure}
\FloatBarrier
 
A similar method can be applied working with logarithmic transformations, yielding a corollary on construction of geometric Brownian motion. 
 \begin{samepage}
 \begin{cor}[Geometric Brownian Motion] Consider the set-up of Theorem~\ref{thm:wiener} with $M_k=(M_k(t))$ in Skorokhod space $\mathcal{D}[0,T_k]$. Let $\mu\in\R$ and $\sigma\in\R_{\ge0}$ and define \[R_k(t)= \exp[\sigma M_k(t)-\tfrac{1}{2}\sigma^2t],\quad Y_k(t) = Y_0\exp[\mu t]R_k(t)=Y_0\exp[\sigma M_k(t) + (\mu-\tfrac{1}{2}\sigma^2)t].\] Then, $R_k(t)$ is a martingale and $Y_k=(Y_k(t))$ converges in distribution in $\mathcal{D}[0,\infty)$ as $k\rightarrow\infty$ to the geometric Brownian motion $Z=(Z_t)$ with mean and covariance 
\[\E Z_t = Y_0 \exp[\mu t],\quad \Cov(Z_s,Z_t) = Y_0^2 \exp[\mu(s+t)](\exp[\sigma^2(s\wedge t)]-1).\]
 \end{cor}
 \end{samepage}

Similarly, stochastic processes may be constructed from the negative and positive dice. We leave this to future work.

\subsection{Traffic flows}
The random counting measure $N=(\kappa,\nu)$ describes a system of countable independent points. Thus it is appropriate for modeling independent data generated according to law $\nu$. Consider the application of traffic counts. Let $E=[0,1]$ be a finite stretch of road lane conveying free-flowing traffic. Assume the road lane has capacity maximum of $n$ vehicles before physically exhausting the space with vehicles. If we assume the minimum number of vehicles conveyed is zero, $m=0$, then we obtain a counting distribution supported on $\{0,1,\dotsb,n\}$. Let $K\sim\kappa_{m,n}$ be the number of vehicles conveyed on the road, and let $\mathbf{X}=\{X_i:i=1,\dotsb,K\}$ be the locations of the vehicles on the road lane in $[0,1]$, independently identically drawn $X_i\sim\nu=\text{Uniform}[0,1]$. The pair $(\kappa_{m,n},\nu)$ forms a random counting measure $N$ on $(E,\mathscr{E})$. Letting $A=[0,x]$ and $B=(x,1]$ be a partition of the road lane into two parts, we obtain covariance of the number of vehicles in the parts $N(A)$ and $N(B)$ as \[\Cov(N(A),N(B))=\frac{n(n-4)}{12}x(1-x).\] Thus if $n=4$, the counts on the road parts are de-correlated, whereas if $n<4$ they are negatively correlated and otherwise $n>4$, they are positively correlated. The correlation is maximized with $x=1/2$, corresponding to equal sized road parts. Thus, the length scale of a segment of road that is sometimes unoccupied (uniformly at random) and exhibits de-correlated counts is precisely four vehicles in length, conveying Poisson-like flows. If the maximum number of vehicles is greater than four, such as for large intersections, then the vehicle count is over-dispersed and the counts in distinct road parts exhibit positive covariance, similar to the negative binomial distribution. As the Poisson-type distributions (binomial, Poisson, negative binomial) are commonly encountered in traffic flows \citep{traffic,traffic2}, their dice counterparts seem naturally suitable for their bounded realities.

These ideas carry over to marked random measures. Consider dice random measure $N=(\kappa_{m,n},\nu\times\pi)$ on $([0,1]\times\R_{\ge0},\mathscr{B}_{[0,1]}\otimes\mathscr{B}_{\R_{\ge0}})$, where each vehicle $i$ at location $X_i$ is independently marked with a velocity $V_i$ according to law $\pi$ with mean $a$ and variance $b^2$. Consider test function $f_A(x,y)=y\ind{A}(x)$ for $A\in\mathscr{E}$. Then $Nf_A$ is the sum of the velocities of the vehicles in road part $A$. We have mean-covariance\begin{align*}\E Nf_A &=\frac{m+n}{2}\nu(A)a\\\Cov(Nf_A,Nf_B)&=\frac{m+n}{2}\nu(A\cap B)(a^2+b^2) + \left(\frac{1}{12} \left(m^2-2 m (n+4)+(n-4) n\right)\right)\nu(A)\nu(B)a^2. \end{align*} Therefore, \[A=[0,x],\quad B=(x,1]\quad\Rightarrow\quad\Cov(Nf_A,Nf_B)=\left(\frac{1}{12} \left(m^2-2 m (n+4)+(n-4) n\right)\right)x(1-x)a^2,\] having similar structure to the non-marked representation. Now consider for vehicle $i$ its location at time $t$, $L_i(t)=X_i+V_it$. Setting $f_{A}(x,y)=(x+ty)\ind{A}(x)$, we have mean-covariance of the random sum of locations  for disjoint parts as \begin{align*}\E Nf_A &=\frac{m+n}{2}(\E X\ind{A}+a\nu(A) t)\\\text{for disjoint }A\text{ and }B&\text{ in }\mathscr{E}:\\\Cov(Nf_A,Nf_B)&=\left(\frac{1}{12} \left(m^2-2 m (n+4)+(n-4) n\right)\right)(\E X\ind{A}+a\nu(A)t)(\E X\ind{B}+a\nu(B)t).\end{align*} Taking $A=[0,x]$ and $B=(x,1]$, we obtain \[\Cov(Nf_A,Nf_B)=\left(\frac{1}{12} \left(m^2-2 m (n+4)+(n-4) n\right)\right)(\frac{x^2}{2}+axt)(\frac{1}{2}-\frac{x^2}{2}+a(1-x)t).\] These findings underscore the mean-covariance structure of mixed binomial processes formed by dice.

A generalization is consider the product random measure $M=N\times N$ on the unit square $[0,1]\times[0,1]$, where $N=(\kappa_{m,n},\nu)$ is a random measure on $([0,1],\mathscr{B}_{[0,1]})$ with $\kappa_{m,n}\in\mathcal{K}$, formed as \[Mf = \int_{[0,1]\times[0,1]}M(\D x,\D y)f(x,y) = \sum_i^K\sum_j^Kf\circ (X_i,X_j)\] where $K\sim\kappa_{m,n}$ with mean $c$ and variance $\delta^2$. By Theorem~\ref{thm:product}, we obtain \[\E Mf = c(\nu\times I)f + (c^2+\delta^2-c)(\nu\times\nu)f\for f\in(\mathscr{B}_{[0,1]\times[0,1]})_{\ge0}.\] Suppose we wanted to know the number of pairs of distinct vehicles that are within some distance $\varepsilon$ of each other. Therefore $f(x,y)=\ind{}(|x-y|\le\varepsilon)\ind{}(x\ne y)$. We obtain mean number as \[\E Mf = (c^2+\delta^2-c)(\nu\times\nu)f = (c^2+\delta^2-c)(2\varepsilon-\varepsilon^2).\] Hence $2\varepsilon-\varepsilon^2$ is the fraction of the mean number of pairs of distinct vehicles (numbering $\E K^2-\E K = c^2+\delta^2-c$) that are within $\varepsilon$ of each other. 

\subsection{Self-similarity, sparse lattices of naturals, and Monte Carlo} In the last application we briefly discuss self-similarity of the orthogonal dice, c.f. Theorem~\ref{thm:selfsimilar}: any orthogonal die of composite size is equivalent to component dice of coprime sizes. For example, consider the first composite die, $\kappa_{m,n}$ for $m=85$ and $n=119$, having size $n-m+1=35=5\times 7$. Thus $\kappa_{m,n}$ may be represented in terms of the first two (primitive) dice $\kappa_{0,4}$ and $\kappa_{1,7}$, or equivalently, for $K_{35}\sim\kappa_{m,n}$, $K_5\sim\kappa_{0,4}$ and $K_7\sim\kappa_{1,7}$, we have $K_{35}\Leftrightarrow(K_5,K_7)$, shown below in Table~\ref{tab:decomp}.

\begin{table}[h!]
\begin{center}
\begin{tabular}{ccc|ccc|ccc|ccc|ccc|}
\toprule
$K_5$ & $K_7$ &$K_{35}$ & $K_5$ & $K_7$ &$K_{35}$ & $K_5$ & $K_7$ &$K_{35}$ & $K_5$ & $K_7$ &$K_{35}$ &$K_5$ & $K_7$ &$K_{35}$  \\\midrule
0& 1& 85&1& 2& 86& 2& 3& 87& 3& 4& 88& 4& 5& 89\\
0& 6&   90& 1& 7& 91&2& 1& 92&3& 2& 93&4& 3& 94\\
0& 4& 95 &1&  5& 96&2& 6& 97&3& 7& 98&4& 1& 99\\
0& 2& 100&1& 3&   101&2& 4& 102&3& 5& 103&4& 6& 104\\
0& 7& 105&1& 1& 106&2& 2& 107&3& 3& 108&4& 4& 109\\
0& 5& 110&1& 6&111&2& 7& 112&3& 1& 113&4& 2& 114\\
0& 3& 115&1& 4&  116&2& 5& 117&3& 6& 118&4& 7& 119\\
\bottomrule
\end{tabular}\caption{Decomposition of $\kappa_{m,n}$ for $m=85$ and $n=119$}\label{tab:decomp}
\end{center}
\end{table}

Let the table form the array $\mathbf{K}=\{(K_5,K_7,K_{35})\}$ with size $|\mathbf{K}|=35$. Each column (marginal) has uniform distribution $\kappa_{0,4},\kappa_{1,7},\kappa_{85,119}$. The first two columns are mutually independent, and the third column is dependent upon the first two. Thus the self-similarity property naturally, sparsely, and uniformly `fills' the space $F=\{0,\dotsb,4\}\times\{1,\dotsb,7\}\times\{85,\dotsb,119\}$ of size $35^2$ using only $35$ points, a square root reduction. Note that $\Var K_5 = 2$, $\Var K_7=4$, $\Var K_{35}=102$, $\Cov(K_5,K_7)=0$, $\Cov(K_5,K_{35})=\Var K_{5}=2$, and $\Cov(K_7,K_{35})=\Var K_{7}=4$. This gives mean vector $\mu$ and covariance $\Sigma$ and correlation $\rho$ matrices, noting $\mu=\text{diag}(\Sigma)$ by construction: \[\mu=\begin{bmatrix}2&4&102
\end{bmatrix},\quad\Sigma =\begin{bmatrix}2 & 0 & 2\\0 & 4 & 4\\2 & 4 & 102\end{bmatrix},\quad\rho =\begin{bmatrix}1 & 0 & \frac{1}{\sqrt{51}}\\0 & 1 & \sqrt{\frac{2}{51}}\\\frac{1}{\sqrt{51}} & \sqrt{\frac{2}{51}} & 1\end{bmatrix}.\] Putting $\kappa=\text{Uniform}(F)$ and $\kappa'=\text{Uniform}(\mathbf{K})$, we compute mean values of various functions below in Table~\ref{tab:decomp2}. Notice that additive functions $f(x,y,z)=g(x)+h(y)+j(z)$ are exactly integrated using $\kappa'$ following the exact uniform marginals of $\kappa'$. Also general functions in the first two coordinates are exactly integrated, following from the independent uniform representation of the first two coordinates. Thus, we have exact representation for functions of the form $f(x,y,z)=g(x,y)+h(z)$. 

\begin{table}[h!]
\begin{center}
\begin{tabular}{lll}
\toprule
Function $f(x,y,z)$ & $\kappa(f)$ & $\kappa'(f)$\\\midrule
$x+y+z$ & 108 & 108\\
$\sqrt{x+y+z}$ & $10.380\dotsb$ & $10.378\dotsb$\\
$xy+z$ & 110 & 110\\
$x^2+y^2+z^2$ & 10\,532 & 10\,532\\
$(x+y+z)^2$ & 11\,772 & 11\,784\\
$(x-y)^2(z-xy)^2$ & 956 & 964\\
\bottomrule
\end{tabular}\caption{Exact and estimated average values of various functions on the cube $F=\{0,\dotsb,4\}\times\{1,\dotsb,7\}\times\{85,\dotsb,119\}$ where $\kappa=\text{Uniform}(F)$ and $\kappa'=\text{Uniform}(\mathbf{K})$}\label{tab:decomp2}
\end{center}
\end{table} 

Scatter plots of $\mathbf{K}$ are shown below in Figure~\ref{fig:scatter}, illustrating the uniform structure.

\begin{figure}[h!]
\centering
\begingroup
\captionsetup[subfigure]{width=3in,font=normalsize}
\subfloat[$(K_5,K_7)$]{\includegraphics[width=3in]{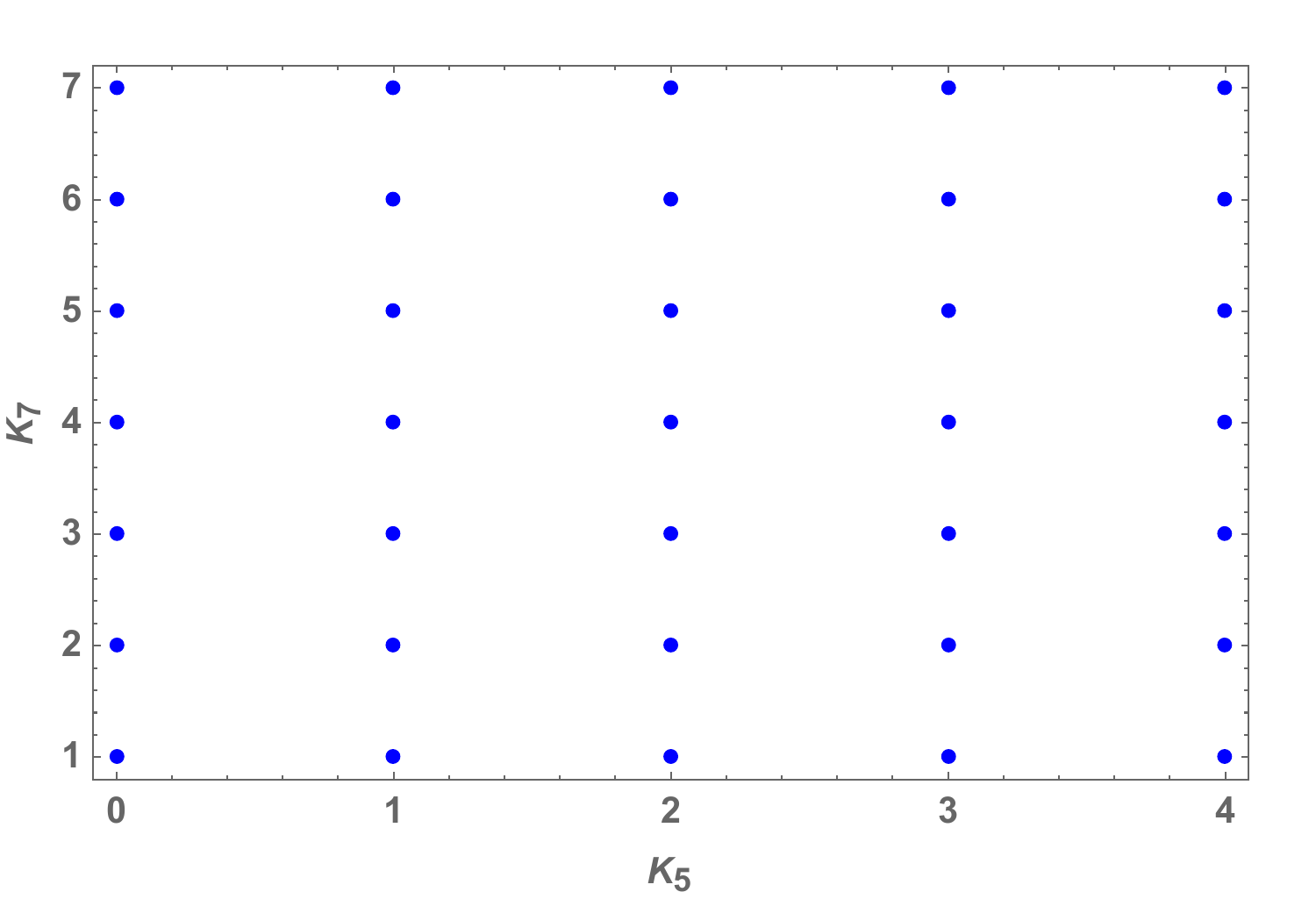}}
\subfloat[$(K_5,K_{35})$]{\includegraphics[width=3in]{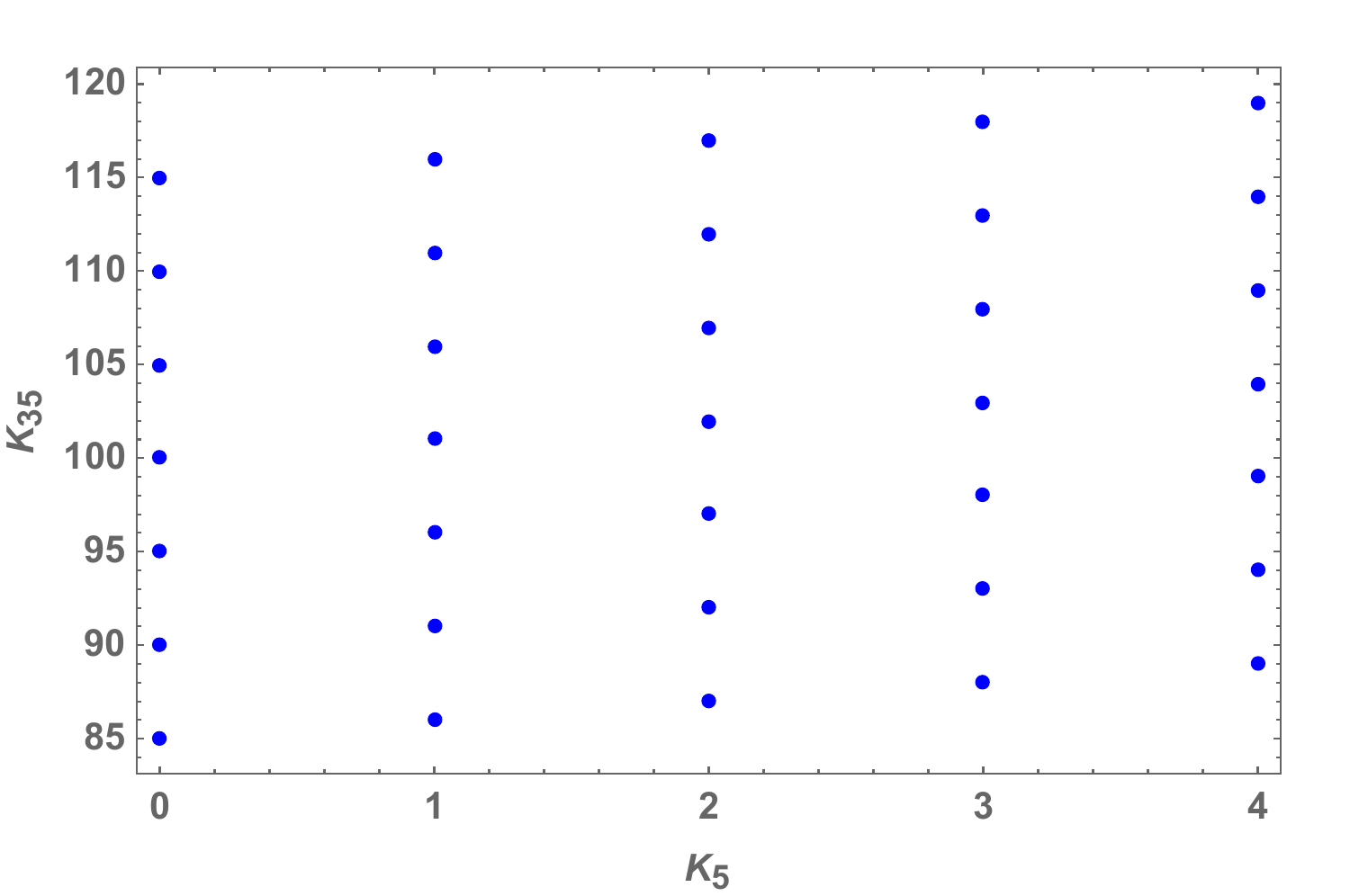}}\\
\subfloat[$(K_7,K_{35})$]{\includegraphics[width=3in]{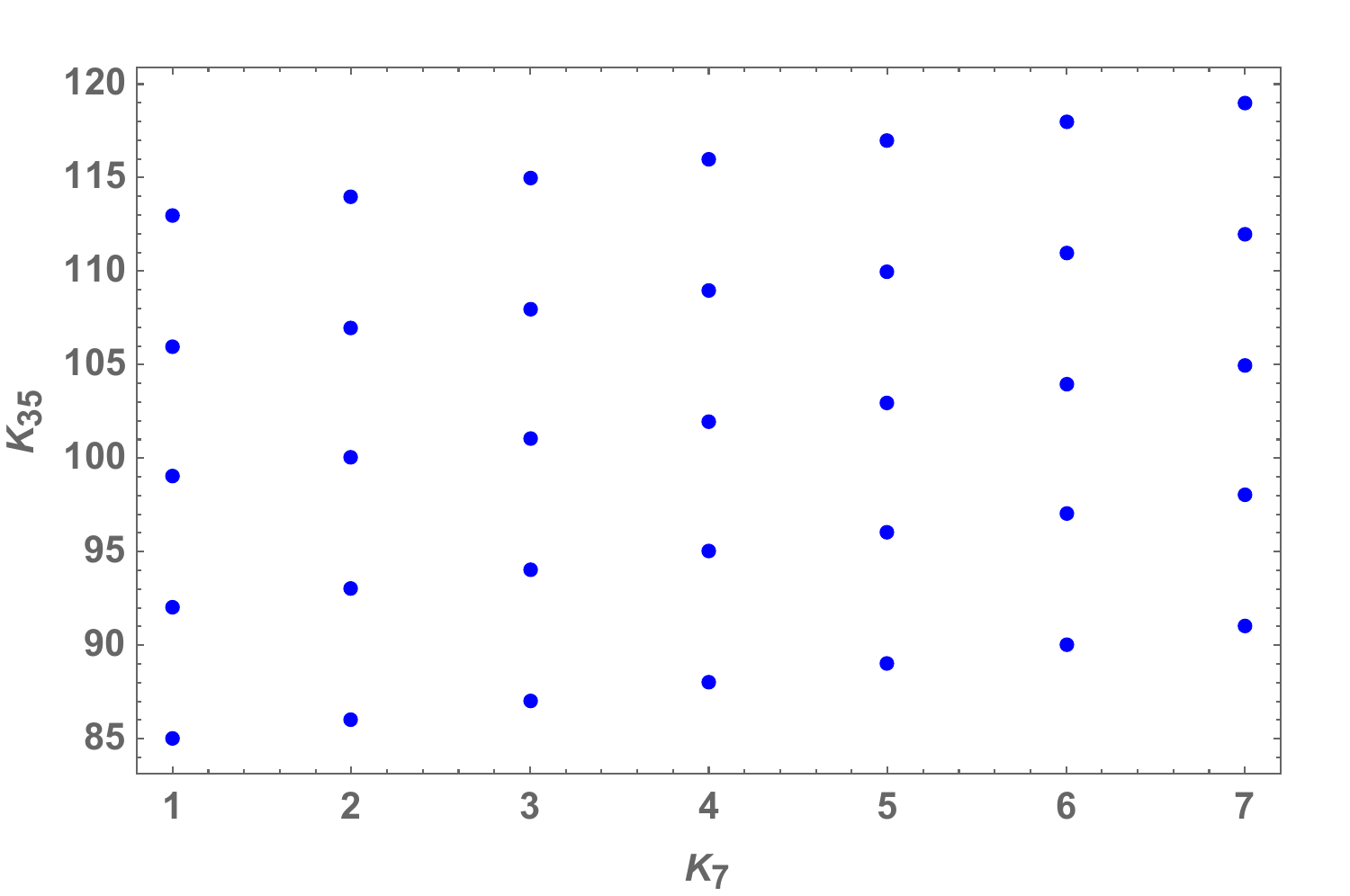}}
\subfloat[$(K_5,K_7,K_{35})$]{\includegraphics[width=3in]{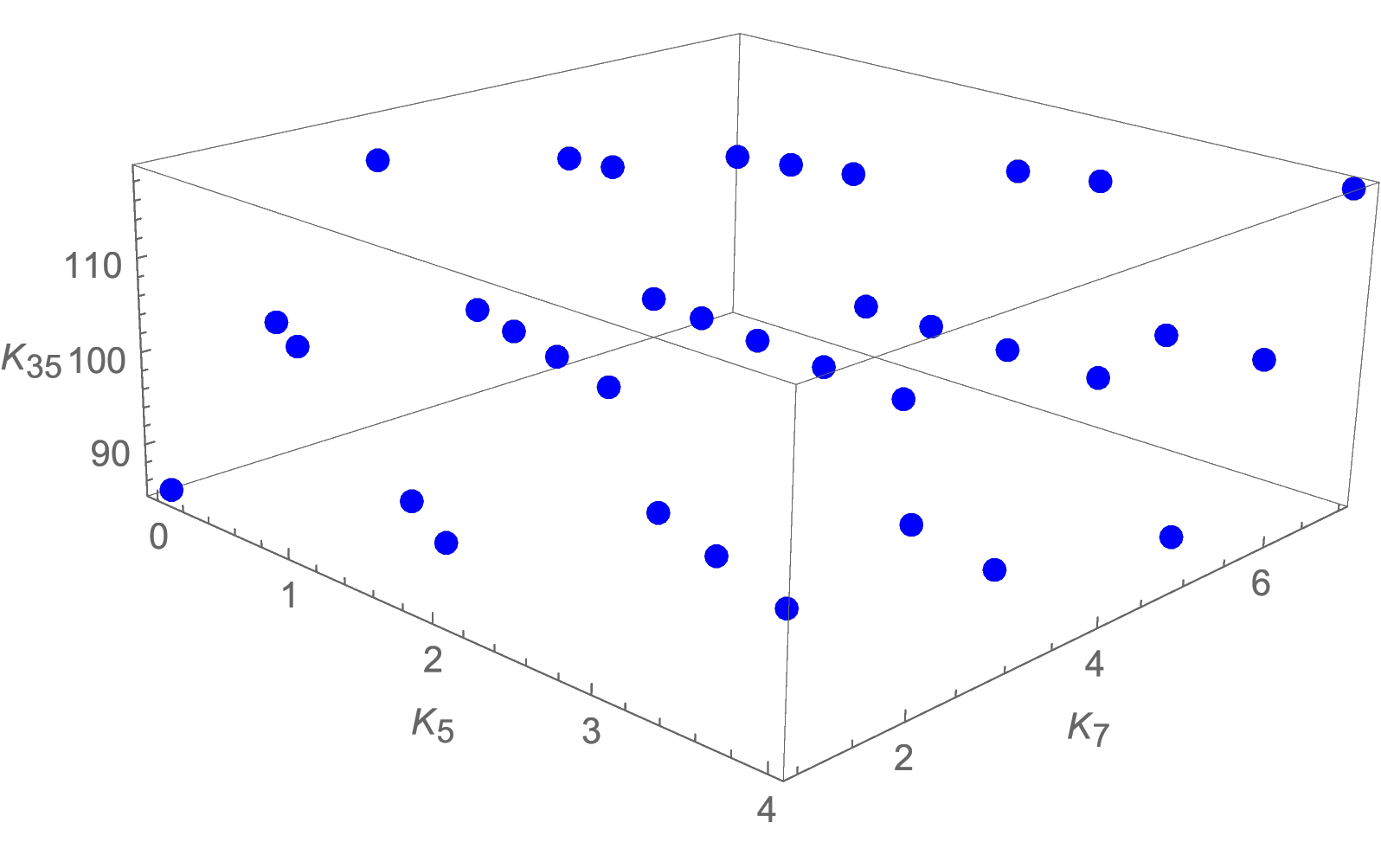}}\\
\endgroup
\caption{Scatter plots of $\mathbf{K}=\{(K_5,K_7,K_{35})\}$ for $m=85$ and $n=119$}\label{fig:scatter}
\end{figure}

\FloatBarrier

In this framework, the coordinates used in Monte Carlo integration differ in the sizes of their grids—by at least a factor of two—which is an inherent feature of the construction. Nevertheless, these variates can be rescaled to lie within the unit interval. In general, a space of size 
$l$, where is coprime to both two and three, is represented using approximately $\sqrt{l}$ points, also coprime to two and three. The full space is thus organized as a square composed of such coprime factors. The final coordinate, referred to as the long coordinate, is defined as the product of the sizes of the preceding short coordinates. It is therefore natural to assign the long coordinate to a uniform variable, such as time or spatial location, while the short coordinates may be used to represent covariates. This construction yields a sparse lattice on the natural numbers that is well-suited for Monte Carlo integration in contexts such as time-series models with time-varying covariates.

Consider the joint representation $(K_1,\dotsb,K_l,K)$ as \[(K_1,\dotsb,K_l)\sim\nu=\times_{i=1}^l\kappa_{m_i,n_i}\] with \[Q(x=(x_1,\dotsb,x_l),y)=\delta_{\text{CRT}(x|b)}(y).\] Then the joint representation is distributed \[(K_1,\dotsb,K_l,K)\sim\mu=\nu\times Q.\] Taking projection mapping $f(x,y)=y$ for the marginal, the marginal law is the composite die, $\mu\circ f^{-1}=\kappa_{m,n}$. The product marginal law is given by \[\mu_{m,n}=\times_{i=1}^l\kappa_{m_i,n_i}\times\kappa_{m,n}.\] Therefore, for all additive functions $f(x,y)=g(x)+h(y)$, we have $\mu(f)=\mu_{m,n}(f)$. 
 
  \section{Discussion and Conclusions}\label{sec:dis} 

  Discrete rectangular uniform distributions—referred to as ``dice''—provide a natural class of models for bounded count data. We classify these distributions into three infinite families based on their dispersion properties: negative, orthogonal, and positive. Orthogonal dice, defined by the equality of mean and variance, lie on the boundary between under- and over-dispersed cases and serve as a reference point in this classification.

We show that there are infinitely many orthogonal dice and examine their probabilistic and arithmetic properties. In particular, their associated random measures converge to the Poisson random measure and can be used to construct orthogonal random measures and Poisson superpositions. In the scaling limit, these measures converge in distribution to classical stochastic processes, including Brownian motion and geometric Brownian motion in the Skorokhod space.

From an arithmetic perspective, orthogonal dice have support sizes that are coprime to two and three. They generate partitions of the natural numbers and decompose the supports of dice into subsets with specific coprimality properties. Additionally, they exhibit a form of self-similarity: dice with composite support sizes can be built from those with coprime sizes.

The variance-to-mean relationship offers a practical criterion for organizing the family of dice, with implications for modeling count data with different dispersion patterns. Since dice distributions can be convolved freely, this framework is flexible and broadly applicable to problems involving bounded discrete outcomes. Further study is warranted to explore connections with other classes of discrete distributions and to develop inference methods tailored to this family. The structure and versatility of the dice distributions suggest that they may serve as useful building blocks in both applied modeling and theoretical investigations.

%If your paper includes appendices, then precede the first of them by the command
\appendix
%and then carry on using the \section and \subsection commands, as above.

\section{Rebolledo's Theorem}\label{appn} We give Rebolledo's theorem, adapted from Appendix C of \cite{Khudabukhsh:2022aa}. 

\begin{thm}[Rebolledo's theorem] Let $M_n$ be a  locally square-integrable martingale. Consider the following two conditions: \begin{align*}\langle M_n\rangle_t&\xrightarrow{P}V(t),\quad t\in\mathcal{T}\\\langle M_n^\varepsilon\rangle_t&\xrightarrow{P}0\quad\quad \forall t\in\mathcal{T},\quad\forall\varepsilon>0\end{align*} where $V$ is a continuous deterministic function on $\mathcal{T}_0$ and $M_n^\varepsilon$ contains the jumps of $M_n$ of size exceeding $\varepsilon$. Then these conditions imply the following finite-dimensional convergence: \[(M_n(t_1),M_n(t_2),\dotsb,M_n(t_l))\xrightarrow{D}(W(t_1),W(t_2),\dotsb,W(t_l))\] as $n\rightarrow\infty$, where $W$ is a Gaussian martingale with $\langle W\rangle=V$ and $t_1,t_2,\dotsb,t_l\in\mathcal{T}$.

If, in addition, $\mathcal{T}$ is dense in $\mathcal{T}_0$, then these conditions imply the following weak convergence: \[M_n\xrightarrow{D}W\] as $n\rightarrow\infty$ in $\mathcal{D}$, the space of $\R$-valued c\`{a}dl\`{a}g functions on $\mathcal{T}_0$ endowed with the Skorokhod topology. Furthermore, $\langle M_n\rangle$ converges uniformly on compact subsets of $\mathcal{T}_0$ to $V$ in probability.
\end{thm}

%If you include EPS (encapsulated postscript) figures in your paper,
%then please use the following commands:
%\begin{figure}
%\begin{center}
%\includegraphics{.eps}
%\caption{Caption text.}\label{}
%\end{center}
%\end{figure}

%%%%%%%%%%Declarations%%%%%%%%%%

% \acks % Place the text of your acknowledgements after the \acks (or \Acks) command. This will generate the heading "Acknowledgements". If you wish to make only one acknowledgement, use \ack (or \Ack).
% We wish to thank...

\fund % Place any funding information for this work after the \fund (or \Fund) command.
C.D.B. acknowledges support from Massive Dynamics (Princeton, New Jersey). H.R. acknowledges support from the Army Research Office (grant number W911NF-19-1-0382). G.R. acknowledges funding from the National Science Foundation (grant number DMS-1853587). 

\competing % Place any information on competing interests after the \competing (or \Competing) command.
There were no competing interests to declare which arose during the preparation or publication process of this article.

% \data % Place any information on data related to the work in your article after the \data (or \Data) command. Omit this command/section and text if it is not relevant to your article.
% The data related to the simulations found in Section 2 can be found at...

% \support The supplementary material for this article can be found at http://doi.org/10.1017/[TO BE SET]. % Delete this line if there are no supplementary files related to this article. If there are supplementary files related to your article, leave the line unchanged.

%%%%%%%%%%%%Reference list%%%%%%%%%%%%%%
% References should be in the following form (or the BibTeX file
% apt.bst should be used):
%
% For a journal:
% Surname, Initial (year). Title of paper. {\em Journal title}
% {\bf Vol,} page--range.
%
% For a book:
% Surname, Initial (year). {\em Book title}. Publisher, Address.
%
% Note the following example of a reference list.

\bibliographystyle{plainnat} % Style BST file (imsart-number.bst or imsart-nameyear.bst)
\bibliography{stones3}  
% \begin{thebibliography}{99}
% \bibliographystyle{APT}
% \footnotesize 

% \bibitem{ref1}
% {\sc Ball, K. and Chain, H.} (1988). {\em Kurtosis: A Critical
% Review}, 2nd~edn. John Wiley, New York.

% \bibitem{ref2}
% {\sc Boyd, W.} (1978). Hyperbolic distributions. Doctoral Thesis,
% University of Boston School of Mathematics.

% \bibitem{ref3}
% {\sc Sichel, H.~S., Kleingeld, W.~J. and Assibey-Bonsu, W.}
% (1992).  A comparative study of three frequency-distribution
% models for use in ore valuation. {\em J. S. Afr. Inst. Min. Met.}
% {\bf 92,} 91--99.

% \end{thebibliography}

\end{document}